\documentclass[final,reqno,onefignum,onetabnum]{siamltex1213}

\usepackage{amsmath,amstext,amsbsy,amssymb,booktabs}
\usepackage{relsize}
\usepackage{todonotes,overpic,mathdots}

\DeclareMathOperator{\adj}{adj}
\newtheorem{example}[theorem]{Example}

\newcommand\coolrightbrace[2]{%
\left.\vphantom{\begin{matrix} #1 \end{matrix}}\right\}#2} 
\newcommand{\cabs}{\kappa_{\mathrm{abs}}}
\newcommand{\crel}{\kappa_{\mathrm{rel}}}

\title{Numerical instability of resultant methods for multidimensional rootfinding} 

\author{Vanni Noferini\thanks{Department of Mathematical Sciences, University of Essex, Wivenhoe Park, Colchester, UK, CO4 3SQ. (\texttt{vnofer@essex.ac.uk})} 
\and Alex Townsend\thanks{Department 
of Mathematics, Massachusetts Institute of Technology, 77 Massachusetts Avenue
Cambridge, MA 02139-4307. (\texttt{ajt@mit.edu})}}

\begin{document}
\maketitle
\slugger{sinum}{xxxx}{xx}{x}{x--x}

\begin{abstract}
Hidden-variable resultant methods are a class of algorithms for solving 
multidimensional polynomial rootfinding problems. In two dimensions, when 
significant care is taken, they are competitive practical rootfinders. However, in higher
dimensions they are known to miss zeros, calculate roots to low precision, 
and introduce spurious solutions. We show that  
the hidden variable resultant method based on  
the Cayley (Dixon or B\'{e}zout) matrix is inherently 
and spectacularly numerically unstable by a factor that grows exponentially with 
the dimension. We also show that the Sylvester matrix for solving
bivariate polynomial systems can square the condition number of the 
problem. 
In other words, two popular hidden variable resultant methods are numerically 
unstable, and this mathematically explains the difficulties that 
are frequently reported by practitioners. Regardless of how the constructed 
polynomial eigenvalue problem is solved, 
 severe numerical difficulties will be present.
Along the way, we prove that the Cayley resultant is a generalization of Cramer's rule for solving linear systems 
and generalize Clenshaw's algorithm to an evaluation scheme for polynomials 
expressed in a degree-graded polynomial basis.
\end{abstract}



\pagestyle{myheadings}
\thispagestyle{plain}
\markboth{}{}

\section{Introduction}
Hidden variable resultant methods are a popular class of algorithms 
for global multidimensional rootfinding~\cite{Allgower_92_01,Dressen_12_01,Jonsson_05_01,Nakatsukasa_15_01,Sorber_14_01,Syam_04_01}. They
compute all the solutions to zero-dimensional polynomial systems of the form: 
\begin{equation}
 \begin{pmatrix}  
  p_1(x_1,\ldots,x_d) \cr
  \vdots \cr
  p_d(x_1,\ldots,x_d)\cr
 \end{pmatrix}
  = 0, \qquad (x_1,\ldots,x_d)\in\mathbb{C}^d,
\label{eq:polynomialSystem}
\end{equation} 
where $d\geq 2$ and $p_1,\ldots,p_d$ are polynomials in $x_1,\ldots,x_d$ with complex coefficients. Mathematically, they are based on an elegant idea that 
converts the multidimensional 
rootfinding problem in~\eqref{eq:polynomialSystem} into one or more eigenvalue problems~\cite{Bondyfalat_00_01}. 
At first these methods appear to be a practitioner's 
dream as a difficult rootfinding problem is solved by the robust QR or QZ
algorithm. Desirably, these methods have received 
considerable research attention from the scientific computing 
community~\cite{Buse_05_01,Emiris_94_01,Li_06_01,Weiss_93_01}.

Despite this significant interest,
hidden variable resultant methods are 
notoriously difficult, if not impossible, to make numerically robust. Most naive 
implementations will introduce unwanted spurious solutions, compute roots inaccurately, and 
unpredictably miss zeros~\cite{Boyd_14_01}. Spurious solutions can be 
removed by manually checking that all 
the solutions satisfy~\eqref{eq:polynomialSystem}, inaccurate 
roots can usually be polished by Newton's method, but entirely missing a zero is 
detrimental to a global rootfinding algorithm. 

The higher the polynomial degree $n$ and the dimension
$d$, the more pronounced the numerical difficulties become. Though 
our conditioning bounds do hold for small $n$ and $d$, 
this paper deals with a worst-case analysis. Hence, our conclusions are
 not inconsistent with 
the observation that (at least when $n$ and $d$ are small) resultant methods 
can work very well in practice for some problems. 
When $d = 2$ and real finite solutions are of interest, a careful combination of domain subdivision, 
regularization, and local refinement has been successfully used together with 
the Cayley resultant (also known as the Dixon or B\'{e}zout resultant)
for large $n$~\cite{Nakatsukasa_15_01}. This is the algorithm employed by 
Chebfun for bivariate global rootfinding~\cite{Townsend_13_01}. Moreover, for $d=2$, randomization techniques 
and the QZ algorithm have been combined fruitfully with the Macaulay resultant~\cite{Jonsson_05_01}. 
There are also many other ideas~\cite{Bini_06_01,Demmel_94_01}.
However, these techniques seem to be less successful in higher dimensions.

In this paper, we show that any plain vanilla hidden variable resultant 
method based on the Cayley or Sylvester matrix is a
numerically unstable algorithm for solving a polynomial system. In particular, we show 
that the hidden variable resultant method based on the 
Cayley resultant matrix is numerically unstable for multidimensional 
rootfinding with a factor that grows exponentially with $d$. We show 
that for $d=2$ the Sylvester matrix leads to a hidden variable 
resultant method that can also square the conditioning of a root. 

We believe that this numerical instability has not been analyzed before 
because there are at least two 
other sources of numerical issues: 
(1) The hidden variable resultant method is usually employed with the monomial polynomial basis, which 
can be devastating in practice when $n$ is large, and (2) Some rootfinding 
problems have inherently ill-conditioned zeros and hence, one does not always 
expect accurate solutions. Practitioners can sometimes overcome (1) by 
representing the polynomials $p_1,\ldots,p_d$ in another degree-graded polynomial basis\footnote{A polynomial basis
$\{\phi_0,\ldots,\phi_{n}\}$ for $\mathbb{C}_{n}[x]$ is degree-graded if the degree of $\phi_k(x)$ is 
exactly $k$ for $0\leq k\leq n$.}~\cite{Boyd_14_01}. However, the numerically instability 
that we identify can be observed even when the roots are well-conditioned and 
for degree-graded polynomial basis (which includes the monomial, Chebyshev, and 
Legendre bases). 

We focus on the purely numerical, as opposed to symbolic, algorithm. We take 
the view that every arithmetic operation is performed in finite precision. 
There are many other rootfinders that either employ only symbolic manipulations~\cite{Buchberger_65_01} 
or some kind of symbolic-numerical hybrid~\cite{Emiris_02_01}. Similar careful symbolic manipulations 
may be useful in overcoming the numerical instability that 
we identify. For example, it may be possible to somehow 
transform the polynomial system~\eqref{eq:polynomialSystem} into 
one that the resultant method treats in a numerical stable manner. 

This paper may be considered as a bearer of bad news. Yet, we take the opposite 
and more optimistic view. We are intrigued by the potential 
positive impact this paper could have on 
rootfinders based on resultants since once a numerical 
instability has been identified the community is much better placed to 
circumvent the issue.

We use the following notation. The space of univariate polynomials with complex 
coefficients of degree at most $n$ is denoted by $\mathbb{C}_{n}[x]$, the space 
of $d$-variate polynomials of maximal degree $n$ in the variables $x_1,\ldots,x_d$ is denoted by 
$\mathbb{C}_n[x_1,\ldots,x_d]$, and if $\mathcal{V}$ is a vector space then the 
Cartesian product space $\mathcal{V}\times\cdots\times\mathcal{V}$ (d-times) is 
denoted by $(\mathcal{V})^d$. Finally, we use ${\rm vec( V )}$ to be the 
vectorization of the matrix or tensor $V$ to a column vector (this is equivalent 
to \texttt{V(:)} in MATLAB). 

Our setup is as follows. First, we suppose that a degree-graded polynomial basis 
for $\mathbb{C}_{n}[x]$, denoted by $\phi_0,\ldots,\phi_{n}$, has been 
selected. All polynomials will be represented using this basis. Second, a 
region of interest $\Omega^d\subset \mathbb{C}^{d}$ is chosen such 
that $\Omega^d$, where 
$\Omega^d$ is the tensor-product domain $\Omega\times \cdots\times\Omega$ ($d$ times), 
contains all the roots that would like to be computed accurately. 
The domain $\Omega\subset\mathbb{C}$ can be a real interval or a bounded 
region in the complex plane.
Throughout, we suppose that $\sup_{x\in\Omega}|\phi_k(x)| =1$ for $0\leq k\leq n$, which is 
a very natural normalization.

Our two main results are in Theorem~\ref{thm:condTheorem} and Theorem~\ref{thm:condSylvester}.  Together they show that there exist $p_1,\ldots,p_d$ in~\eqref{eq:polynomialSystem} 
such that 
\[
\underbrace{\kappa(x_d^{\ast}, R)}_{\text{Cond.~no.~of the eigenproblem}} \geq (\underbrace{\|J(\underline{x}^{\ast})^{-1} \|_2}_{\text{Cond.~no.~of $\underline{x}^{\ast}$}})^d,
\]
where $R$ is either the Cayley (for any $d\geq2$) or Sylvester (for $d = 2$) resultant matrix. 
Such a result shows that in the absolute sense the eigenvalue problem employed by 
these two resultant-based methods can be significantly more sensitive to perturbations 
than the corresponding root. Together with results about relative conditioning, we conclude that these 
rootfinders are numerically unstable (see Section~\ref{sec:absRel}). 

In the next section we first introduce multidimensional 
resultants and describe hidden variable resultant methods for rootfinding. In 
Section~\ref{sec:Cayley} we show that the hidden variable resultant method based 
on the Cayley resultant suffers from numerical instability and in Section~\ref{sec:Sylvester} 
we show that the Sylvester matrix has a similar instability for $d=2$. 
In Section~\ref{sec:absRel} we explain why our absolute conditioning analysis 
leads to an additional twist when considering relative conditioning. 
Finally, in Section~\ref{sec:futureWork} we present a brief outlook on future 
directions. 

\section{Background material}
This paper requires some knowledge of 
multidimensional rootfinding, hidden variable resultant methods, 
matrix polynomials, and conditioning analysis. In this section we briefly review 
this material.  

\subsection{Global multidimensional rootfinding}
Global rootfinding in high dimensions can be a difficult and computationally expensive 
task.
Here, we are concerned with the easiest situation where~\eqref{eq:polynomialSystem}
has only simple finite roots.
\begin{definition}[Simple root]  
Let $\underline{x}^{\ast} = (x_1^{\ast},\ldots,x_d^{\ast})\in \mathbb{C}^d$ be
a solution to the zero-dimensional polynomial system~\eqref{eq:polynomialSystem}.  
Then, we say that $\underline{x}^{\ast}$ is a simple root of~\eqref{eq:polynomialSystem} 
if the Jacobian matrix $J(\underline{x}^{\ast})$ is invertible, where 
 \begin{equation}
 \qquad J(\underline{x}^{\ast}) = \begin{bmatrix}\frac{\partial p_1}{\partial x_1}(\underline{x}^{\ast}) & \ldots & \frac{\partial p_1}{\partial x_d}(\underline{x}^{\ast})\\[5pt] \vdots & \ddots & \vdots \\[5pt] \frac{\partial p_d}{\partial x_1}(\underline{x}^{\ast}) & \ldots & \frac{\partial p_d}{\partial x_d}(\underline{x}^{\ast})\end{bmatrix}\in\mathbb{C}^{d\times d}. 
 \label{eq:Jacobian}
 \end{equation} 
\end{definition}

If $J(\underline{x}^{\ast})$ is 
not invertible then the problem is ill-conditioned, and 
a numerically stable algorithm working in
finite precision arithmetic may introduce a spurious solution or may 
miss a non-simple root entirely.  We will consider the roots of~\eqref{eq:polynomialSystem} that 
are well-conditioned (see Proposition~\ref{def:conditioning}), finite, and simple. 

Our focus is on the accuracy of hidden variable resultant methods, 
not computational speed. In general, one 
cannot expect to have a ``fast'' algorithm for global multidimensional 
rootfinding. This is because the zero-dimensional polynomial system in~\eqref{eq:polynomialSystem}
can potentially have a large number of solutions. To say exactly how many solutions
there can be, we first must be more 
precise about what we mean by the degree of a polynomial in the multidimensional setting~\cite{Sommese_05_01}.

\begin{definition}[Polynomial degree]
 A $d$-variate polynomial $p(x_1,\ldots,x_d)$ has total degree $\leq n$ if
 \[
  p(x_1,\ldots,x_d) = \sum_{i_1+\cdots+i_d\leq n} A_{i_1,\ldots,i_d} \prod_{k=1}^d \phi_{i_k}(x_k)
 \]
 for some tensor $A$. It is of total degree $n$ if one of the terms 
 $A_{i_1,\ldots,i_d}$ with $i_1+\cdots+i_d= n$ is nonzero. 
 Moreover, $p(x_1,\ldots,x_d)$ has maximal degree $\leq n$ if 
 \[
  p(x_1,\ldots,x_d) = \sum_{i_1,\ldots,i_d=0}^n A_{i_1,\ldots,i_d} \prod_{k=1}^d \phi_{i_k}(x_k)
 \]
 for some tensor $A$ indexed by $0\leq i_1,\ldots,i_d\leq n$. It is of maximal degree $n$ if one of the terms 
 $A_{i_1,\ldots,i_d}$ with $\max(i_1,\ldots,i_d) = n$ is nonzero. 
\end{definition}

B\'{e}zout's Lemma says that if~\eqref{eq:polynomialSystem} involves polynomials of 
total degree $n$, then there are at most $n^d$ solutions~\cite[Chap.~3]{Kirwan_92_01}. 
For polynomials of maximal degree we have the following analogous bound (see also~\cite[Thm.~5.1]{Townsend_14_01}). 
\begin{lemma} 
 The zero-dimensional polynomial system in~\eqref{eq:polynomialSystem}, where $p_1,\ldots,p_d$ 
 are of maximal degree $n$, can have at most $d! n^d$ solutions.
 \label{lem:MaxSolutions}
\end{lemma}
\begin{proof} 
 This is the multihomogeneous B\'{e}zout bound, see~\cite[Thm.~8.5.2]{Sommese_05_01}. 
 For polynomials of maximal degree $n$ the bound is simply ${\rm perm}(nI_d) = d!n^d$, 
 where $I_d$ is the $d\times d$ identity matrix and ${\rm perm}(A)$ is the permanent 
 of $A$. 
%
\end{proof}

We have selected maximal degree, rather than total degree, because 
maximal degree polynomials are more closely linked to tensor-product constructions 
and make later analysis in the multidimensional setting easier.  We do not know 
how to repeat the same analysis when the polynomials are represented in a sparse 
basis set. 

Suppose that the polynomial system~\eqref{eq:polynomialSystem} contains 
polynomials of maximal degree $n$. Then, to verify that $d!n^d$ candidate points are solutions the polynomials 
$p_1,\ldots,p_d$ must be evaluated, costing $\mathcal{O}(n^{2d})$ operations. 
Thus, the optimal worst-case complexity is $\mathcal{O}(n^{2d})$.
For many applications global rootfinding is computationally unfeasible 
and instead local methods such as Newton's method and homotopy continuation 
methods~\cite{Bates_13_01} can be employed to compute a subset of the 
solutions. Despite the fact that 
global multidimensional rootfinding is a computationally intensive task, we 
still desire a numerically stable algorithm. A survey of numerical rootfinders 
is given in~\cite[Chap.~5]{Townsend_14_01}. 

When $d = 1$, global numerical rootfinding can be done satisfactorily 
even with polynomial degrees in the thousands. Excellent 
numerical and stable rootfinders can be built using domain subdivision~\cite{Boyd_02_01}, 
eigenproblems with colleague or comrade matrices~\cite{Good_61_01}, and a 
careful treatment of dynamic range issues~\cite{Boyd_02_01}.

\subsection{Hidden variable resultant methods}
The first step of a hidden variable resultant method is to select 
a variable, say $x_d$, and regard the $d$-variate polynomials $p_1,\ldots,p_d$ in~\eqref{eq:polynomialSystem} 
as polynomials in $x_1,\ldots,x_{d-1}$ with complex coefficients that depend on $x_d$. 
That is, we ``hide'' $x_d$ by rewriting $p_k(x_1,\ldots,x_d)$ for $1\leq k\leq d$ as 
\[
 p_k(x_1,\ldots,x_{d-1},x_d) = p_k[x_d](x_1,\ldots,x_{d-1}) = \sum_{i_1,\ldots,i_{d-1}=0}^{n} c_{i_1,\ldots,i_{d-1}}(x_d)\prod_{s=1}^{d-1}\phi_{i_s}(x_s),
\] 
where $\{\phi_{0},\ldots,\phi_{n}\}$ is a degree-graded polynomial 
basis for $\mathbb{C}_{n}[x]$. 
This new point of view rewrites~\eqref{eq:polynomialSystem} as a system of 
$d$ polynomials in $d-1$ variables. We now seek all the 
$x_d^{\ast}\in\mathbb{C}$ such that $p_1[x_d^{\ast}],\ldots,p_d[x_d^{\ast}]$ 
have a common root in $\Omega^{d-1}$. 
Algebraically, this can be achieved by using a multidimensional resultant~\cite[Chap.~13]{Gelfand_08_01}. 

\begin{definition}[Multidimensional resultant] 
 Let $d\geq 2$ and $n\geq 0$. A functional $\mathcal{R}:(\mathbb{C}_{n}[x_1,\ldots,x_{d-1}])^{d}\rightarrow\mathbb{C}$
 is a multidimensional resultant if, for any set of $d$ polynomials 
 $q_1,\ldots,q_{d}\in\mathbb{C}_{n}[x_1,\ldots,x_{d-1}]$, $\mathcal{R}(q_1,\ldots,q_{d})$ 
 is a 
 polynomial in the coefficients of $q_1,\ldots,q_{d}$ and $\mathcal{R}(q_1,\ldots,q_{d}) = 0$ if and only 
 if there exists an $\underline{x}^\ast\in\tilde{\mathbb{C}}^{d-1}$ such that $q_k(\underline{x}^{\ast})=0$ 
 for $1\leq k\leq d$, where $\tilde{\mathbb{C}}$ denotes the extended complex plane\footnote{To make sense of 
 solutions at infinity one can work with homogeneous polynomials~\cite{Cattani_05_01}.}. 
\label{def:resultantMulti}
 \end{definition}

Definition~\ref{def:resultantMulti} defines $\mathcal{R}$ up to a nonzero multiplicative 
constant~\cite[Thm.~1.6.1]{Cattani_05_01}. In the monomial basis it is standard to normalize $\mathcal{R}$ so that $\mathcal{R}(x_1^n,\ldots,x_{d-1}^n,1)=1$~\cite[Thm.~1.6.1(ii)]{Cattani_05_01}.  For nonmonomial bases, we are not aware of any standard normalization.

Assuming~\eqref{eq:polynomialSystem} only has finite solutions, if 
$\mathcal{R}$ is a multidimensional resultant then for any $x_d^{\ast}\in\mathbb{C}$ we have 
\[
 \mathcal{R}(p_1[x_d^{\ast}],\ldots,p_d[x_d^{\ast}])\!=\!0 \!\quad \!\Longleftrightarrow\!\quad\! \exists (x_1^{\ast},\ldots,x_{d-1}^{\ast})\in\mathbb{C}^{d-1}\text{ s.t. }p_1(\underline{x}^{\ast})\!=\!\cdots\!=\!p_d(\underline{x}^{\ast})\!=\!0, 
\]
where $\underline{x}^{\ast} = (x_1^{\ast},\ldots,x_d^{\ast})\in\mathbb{C}^d$. 
Thus, we can calculate the $d$th component of all the solutions of interest 
by computing the roots of $\mathcal{R}(p_1[x_d],\ldots,p_d[x_d])$ and discarding 
those outside of $\Omega$. 
In principle, since $\mathcal{R}(p_1[x_d],\ldots,p_d[x_d])$ is a univariate 
polynomial in $x_d$ it is an easy task. 
However, numerically, $\mathcal{R}$ is typically near-zero in large regions of $\mathbb{C}$, and 
spurious solutions as well as missed zeros plague this approach in finite precision arithmetic 
(see Figure~\ref{fig:ResultantTrouble}). Thus, directly computing the roots 
of $\mathcal{R}$ is spectacularly numerically unstable for almost all $n$ and 
$d$. This approach is rarely advocated in practice. 
\begin{figure} 
 \centering 
 \begin{overpic}[trim = 0 18 0 0,clip,width=.6\textwidth]{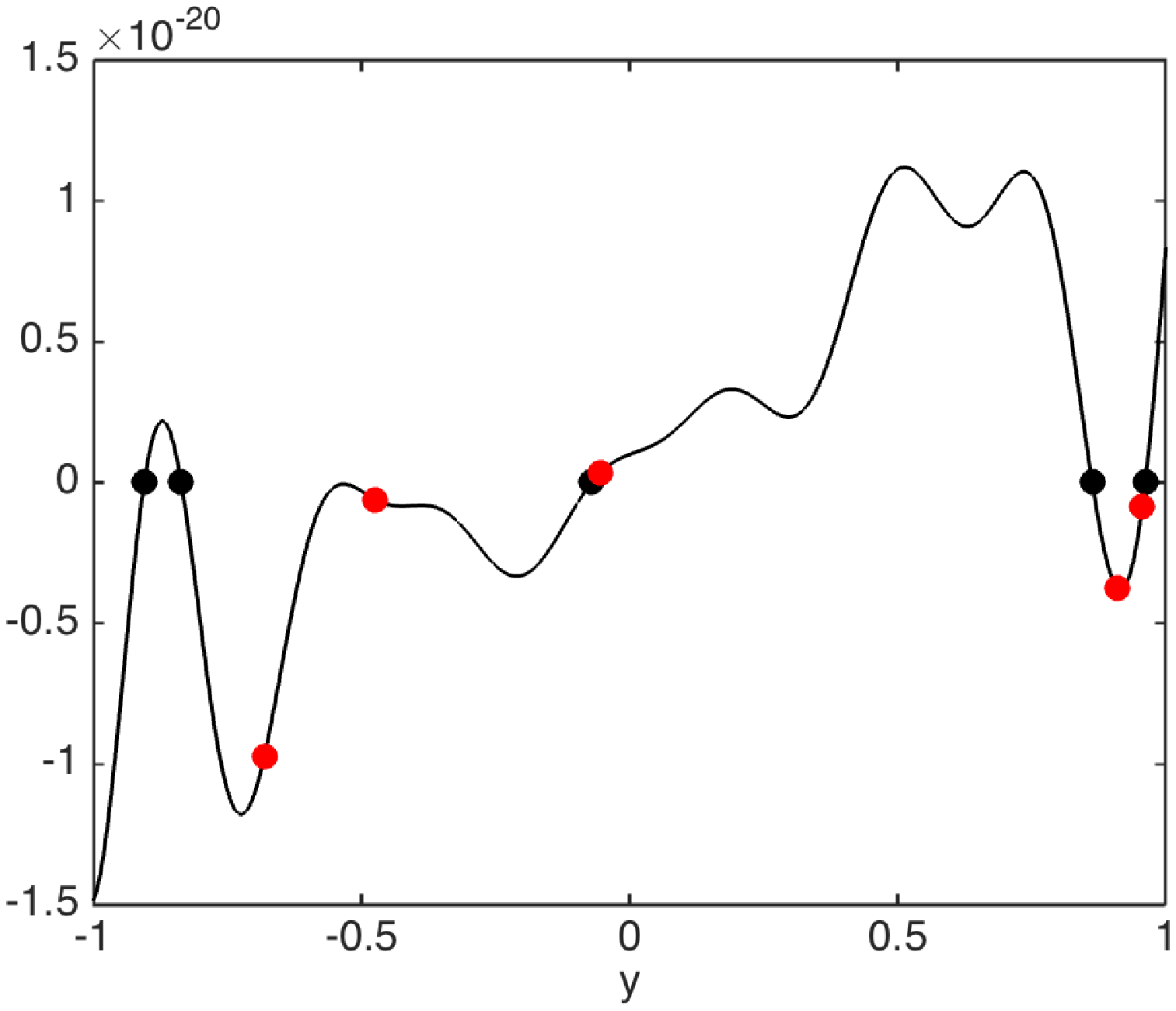}
\put(15,52) {Missed zeros}
\put(30,10) {Spurious solution}
\put(50,25) {Inaccurate root}
\put(30,50) {\vector(-1,-1){7}}
\put(42,15) {\vector(-1,2){7}}
\put(70,29) {\vector(2,1){10}}
\put(50,-2.5) {$x_2$}
\end{overpic} 
 \caption{Mathematically, the zeros of 
 $\mathcal{R}(p_1[x_d],\ldots,p_d[x_d])$ are the $d$th 
 component of the solutions to~\eqref{eq:polynomialSystem}. However, numerically
 the polynomial $\mathcal{R}(p_1[x_d],\ldots,p_d[x_d])$ can be numerically 
 close to zero everywhere. 
 Here, we depict the typical behavior of the polynomial $\mathcal{R}(p_1[x_d],\ldots,p_d[x_d])$ 
 when $d = 2$, where the black dots are the exact zeros and the squares 
 are the computed roots. In practice, it can be difficult to distinguish between spurious solutions and roots that are computed inaccurately.}
 \label{fig:ResultantTrouble} 
\end{figure}

Instead, one often considers an associated multidimensional resultant 
matrix whose determinant is equal to 
$\mathcal{R}$. Working with matrices rather than determinants is 
beneficial for practical computations, especially when $d = 2$~\cite{Dressen_12_01,Nakatsukasa_15_01,Sorber_14_01}. 
Occasionally, this variation on hidden variable resultant methods is called 
{\em numerically confirmed eliminants} to highlight its improved numerical 
behavior~\cite[Sec.~6.2.2]{Sommese_05_01}. However, we will show that even after this significant improvement
the hidden variable resultant methods based on the Cayley and Sylvester resultant matrices remain 
numerically unstable.
\begin{definition}[Multidimensional resultant matrix] 
 Let $d\geq 2$, $n\geq 0$, $N\geq 1$, and $\mathcal{R}$ be a multidimensional resultant (see Defintion~\ref{def:resultantMulti}). A 
 matrix-valued function $R:(\mathbb{C}_{n}[x_1,\ldots,x_{d-1}])^d\rightarrow \mathbb{C}^{N\times N}$ is a
 multidimensional resultant matrix associated with $\mathcal{R}$ if for 
 any set of $d$ polynomials $q_1,\ldots,q_{d}\in\mathbb{C}_{n}[x_1,\ldots,x_{d-1}]$ we have
 \[
  \det \left( R(q_1,\ldots,q_d) \right) = \mathcal{R}(q_1,\ldots,q_d). 
 \]
\end{definition}

There are many types of resultant matrices including Cayley (see Section~\ref{sec:Cayley}), Sylvester (see Section~\ref{sec:Sylvester}), 
Macaulay~\cite{Jonsson_05_01}, and others~\cite{Emiris_94_01,Kapur_95_01,Manocha_92_01}. 
In this paper we only consider two of the most popular choices: Cayley and Sylvester resultant matrices. 


Theoretically, we can calculate the $d$th component of the solutions by finding all the 
$x_d^{\ast}\in\mathbb{C}$ such that $\det(R(p_1[x_d^{\ast}],\ldots,p_d[x_d^{\ast}]))=0$. 
In practice, our analysis will show that this $d$th component cannot always be accurately 
computed. 

Each entry of the matrix $R(p_1[x_d],\ldots,p_d[x_d])$ is a polynomial in $x_d$ of finite degree.
In linear algebra such objects are called matrix polynomials (or polynomial matrices)
and finding the solutions of $\det(R(p_1[x_d],\ldots,p_d[x_d]))=0$ is related to
a polynomial eigenproblem~\cite{Bini_13_01, Mackey_06_01, Tisseur_01_01}. 

\subsection{Matrix polynomials}\label{sec:matrixpoly}
Since multidimensional resultant matrices are matrices with univariate 
polynomial entries, matrix polynomials
 play an important role in the 
hidden variable resultant method.  A classical reference on matrix polynomials 
is the book by Gohberg, Lancaster, and Rodman~\cite{Gohberg_82_01}. 
\begin{definition}[Matrix polynomial]  
 Let $N\geq 1$ and $K\geq 0$. We say that $P(\lambda)$ is a (square) matrix 
 polynomial of size $N$ and degree $K$ if $P(\lambda)$ is an $N\times N$ matrix whose entries 
 are univariate polynomials
 in $\lambda$ of degree $\leq K$, where at least one entry is of degree exactly $K$. 
\end{definition}

In fact, since~\eqref{eq:polynomialSystem}
is a zero-dimensional polynomial system it can only have a finite number of 
isolated solutions and hence, the matrix polynomials we consider 
are regular~\cite{Gohberg_82_01}. 
\begin{definition}[Regular matrix polynomial]  
 We say that a square matrix polynomial $P(\lambda)$ is regular if 
 $\det (P(\lambda))\neq 0$ for some $\lambda\in\mathbb{C}$. 
\end{definition}

A matrix polynomial $P(\lambda)$ of size $N$ and degree $K$ can 
be expressed in a degree-graded polynomial basis as 
\begin{equation} 
 P(\lambda) = \sum_{i=0}^{K} A_i \phi_i(\lambda), \qquad A_i\in\mathbb{C}^{N\times N}.
\label{eq:matPoly}
\end{equation}
When the leading coefficient matrix $A_K$ in~\eqref{eq:matPoly} 
is invertible the eigenvalues of $P(\lambda)$ are all finite, and 
they satisfy $\det(P(\lambda)) = 0$. 
\begin{definition}[Eigenvector of a regular matrix polynomial]
Let $P(\lambda)$ be a regular matrix polynomial of size $N$ and degree $K$. 
If $\lambda\in\mathbb{C}$ is finite and there exists a non-zero vector $v\in \mathbb{C}^{N\times 1}$ such that 
$P(\lambda)v = 0$ (resp.~$v^TP(\lambda)=0$), then we say that  $v$ 
is a {\em right (resp.~left) eigenvector} of $P(\lambda)$ corresponding 
to the eigenvalue $\lambda$. 
\label{def:MatpolyEigenvector}
\end{definition}

For a regular matrix polynomial $P(\lambda)$ we have the following relationship 
between its eigenvectors and determinant~\cite{Gohberg_82_01}: For any finite $\lambda\in\mathbb{C}$, 
\[
\det(P(\lambda)) = 0  \quad\Longleftrightarrow\quad \exists v\in \mathbb{C}^{N\times 1}\setminus\{0\},\quad P(\lambda)v = 0.  
\]

In multidimensional rootfinding, one sets $P(\lambda) = R(p_1[\lambda],\ldots,p_d[\lambda])$ and 
solves $\det(P(\lambda))=0$ via the polynomial eigenvalue problem 
$P(\lambda)v = 0$. There are various algorithms for solving $P(\lambda)v = 0$ including 
linearization~\cite{Gohberg_82_01, Mackey_06_01, Tisseur_01_01}, the Ehrlich--Aberth method~\cite{Bini_13_01, Gemignani_13_01, Taslaman_15_01}, and 
contour integration~\cite{Asakura_10_01}.
However, regardless of how the polynomial eigenvalue problem is
solved in finite precision, the hidden variable resultant method 
based on the Cayley or the Sylvester matrix is numerically unstable.

For the popular resultant matrices, such as Cayley and Sylvester,
the first $d-1$ components of the solutions can be determined from the
left or right eigenvectors of $R(p_1[x_d^{\ast}],\ldots,p_d[x_d^{\ast}])$. For instance, 
if linearization is employed, the multidimensional rootfinding problem is converted into one 
(typically very large) eigenproblem, which can be solved by the QR or QZ algorithm. 
Practitioners often find that the computed eigenvectors are not accurate enough to adequately
determine the $d-1$ components. However, the blame for the observed numerical instability is not only
on the eigenvectors, but also the eigenvalues.  Our analysis will show that the $d$th component
may not be computed accurately either.

\subsection{Conditioning analysis} 
Not even a numerically stable algorithm can be expected to accurately compute a simple 
root of~\eqref{eq:polynomialSystem} if that root is itself sensitive to small 
perturbations. Finite precision arithmetic 
almost always introduces roundoff errors and if these can cause large 
perturbations in a root then that solution is ill-conditioned. 

The absolute condition number of a simple root measures how sensitive the location 
of the root is to small perturbations in $p_1,\ldots,p_d$.
\begin{proposition}[The absolute condition number of a simple root]
 Let $\underline{x}^{\ast} = (x_1^{\ast},\ldots,x_d^{\ast})\in \mathbb{C}^d$ 
 be a simple root of~\eqref{eq:polynomialSystem}. 
 The absolute condition number of $\underline{x}^{\ast}$ associated with rootfinding is 
 $\|J(\underline{x}^{\ast})^{-1}\|_2$, i.e., the matrix $2$-norm of the 
 inverse of the Jacobian.
 \label{def:conditioning}
\end{proposition}
\begin{proof}
See~\cite{Nakatsukasa_15_01}.  
\end{proof}

As a rule of thumb, a numerically stable rootfinder should be able to compute a simple root 
$\underline{x}^{\ast}\in\mathbb{C}^d$ to an accuracy of
$\mathcal{O}(\max(\|J(\underline{x}^{\ast})^{-1}\|_2,1) u)$, where $u$ is the unit machine roundoff.  
In contrast, regardless of the condition number of $\underline{x}^{\ast}$, a numerically unstable rootfinder may not compute it accurately. Worse still, it 
may miss solutions with detrimental consequences. 

A hidden variable resultant method computes the $d$th component of the solutions 
by solving the polynomial eigenvalue problem $R(p_1[x_d],\ldots,p_d[x_d])v = 0$.
The following condition number 
tells us how sensitive an eigenvalue is to small perturbations in $R$~\cite[(12)]{Nakatsukasa_15_01} (also see~\cite{Tisseur_00_01}):
\begin{definition}[The absolute condition number of an eigenvalue of a regular matrix polynomial]
 Let $x_d^{\ast}\in\mathbb{C}$ be a finite eigenvalue of $R(x_d) \!=\! R(p_1[x_d],\ldots,p_d[x_d])$. 
The condition number of $x_d^{\ast}$ associated with the eigenvalue 
problem $R(x_d)v = 0$ is
\begin{equation}
\label{eq:condbezform}
\kappa(x_d^{\ast},R) = \lim_{\epsilon\rightarrow 0^+}\sup\left\{\frac{1}{\epsilon}\min|\hat{x}_d-x_d^{\ast}|:\det(\widehat{R}(\hat{x}_d)) = 0\right\},
\end{equation}
where the supremum is taken over the set of matrix polynomials $\widehat{R}(x_d)$ such that $\max_{x_d\in\Omega} \|\widehat{R}(x_d)-R(x_d)\|_2\leq \epsilon$. 
\label{def:condEigenvalue}
\end{definition}

A numerical polynomial eigensolver
can only be expected to compute the eigenvalue $x_d^{\ast}$ satisfying $R(x_d^{\ast})v=0$ to an accuracy of
$\mathcal{O}(\max(\kappa(x_d^{\ast},R),1)u)$, where $u$ is unit machine roundoff.
We will be interested in how $\kappa(x_d^{\ast},R)$ relates to the condition number, 
$\|J(\underline{x}^{\ast})^{-1}\|_2$, of the corresponding root. 

It can be quite difficult to calculate $\kappa(x_d^{\ast},R)$ directly 
from~\eqref{eq:condbezform}, and is usually more convenient to use the formula below. 
(Related formulas can be found in~\cite[Thm.~1]{Nakatsukasa_15_01} for symmetric matrix 
polynomials and in~\cite[Thm.~5]{Tisseur_00_01} for general matrix polynomials.)
\begin{lemma} 
 Let $R(x_d)$ be a regular matrix polynomial with finite simple eigenvalues. 
 Let $x_d^{\ast}\in\mathbb{C}$ be an eigenvalue of $R(x_d)$ with corresponding 
 right and left eigenvectors $v,w\in\mathbb{C}^{N\times 1}$. Then, we have
 \[
  \kappa(x_d^{\ast},R) = \frac{\|v\|_2\|w\|_2}{|w^TR'(x_d)v|}, 
 \]
 where $R'(x_d)$ denotes the derivative of $R$ with respect to $x_d$. 
 \label{lem:conditionExpression}
\end{lemma}
\begin{proof} 
The first part of the proof follows the analysis in~\cite{Tisseur_00_01}.
Let $R(x_d)$ be a regular matrix polynomial with a simple eigenvalue $x_d^{\ast}\in\mathbb{C}$
and corresponding right and left eigenvectors $v,w\in\mathbb{C}^{N\times 1}$. 
A perturbed matrix polynomial $\hat R(x) = R(x) + \Delta R(x)$ will have a perturbed 
eigenvalue $\hat x_d$ and a perturbed eigenvector $\hat v = v + \delta v$
such that $R(\hat x_d) \hat v + \Delta R(\hat x_d) \hat v=0$, where $\|\Delta R(x)\|_2\leq \epsilon$. 

Expanding, keeping only the first order terms, and using $R(x_d^\ast)v=0$ we obtain
\[ 
(\hat x_d - x_d^\ast) R'(x_d^\ast)  v + R(x_d^\ast) \delta v + \Delta R(x_d^\ast) v = \mathcal{O}(\epsilon^2). 
\]
Multiplying by $w^T$ on the left, rearranging, and keeping the first order terms, we obtain
\[
  \hat{x}_d = x_d^{\ast} - \frac{w^T\Delta R(x_d^{\ast})v}{w^TR'(x_d^{\ast})v},
\]
where the derivative in $R'(x_d^{\ast})$ is taken with respect to $x_d$.
Thus, from~\eqref{eq:condbezform} we see that 
\begin{equation}
 \kappa(x_d^{\ast},R) \leq \frac{\|v\|_2\|w\|_2}{|w^TR'(x_d^{\ast})v|}.
\label{eq:condNum}
\end{equation}
We now show that the upper bound in~\eqref{eq:condNum} can be attained. Take $\Delta R(x_d) = \epsilon w v^T/(\|v\|_2\|w\|_2)$. Then, $\max_{x_d\in\Omega} \|\Delta R(x_d)\|_2= \epsilon$ and 
\[
 \frac{w^T\Delta R(x_d^{\ast})v}{w^TR'(x_d^{\ast})v} = \epsilon\frac{\|v\|_2\|w\|_2}{w^TR'(x_d^{\ast})v}.
\]
The result follows by Definition~\ref{def:condEigenvalue}.
\end{proof}

For the Cayley resultant matrix (see Section~\ref{sec:Cayley}), we will 
show that $\kappa_2(x_d^*,R)$ can be as large as 
$\|J(\underline{x}^{\ast})^{-1}\|_2^{d}$ (see Theorem~\ref{thm:condTheorem}). 
Thus, there can be an exponential increase in the conditioning
that seems inherent to the methodology of the hidden variable resultant 
method based on the Cayley resultant matrix. 
In particular, once
the polynomial eigenvalue problem has been constructed, a backward stable 
numerical eigensolver may not compute accurate solutions to~\eqref{eq:polynomialSystem}. 

We now must tackle the significant 
challenge of showing that the Cayley and Sylvester resultant matrices do lead 
to numerical unstable hidden variable resultant methods, i.e., for certain 
solutions $\underline{x}^{\ast}$ the quantity $\kappa_2(x_d^*,R)$ can be much 
larger than $\|J(\underline{x}^{\ast})^{-1}\|_2$.

\section{The Cayley resultant is numerically unstable for multidimensional rootfinding}\label{sec:Cayley}
The hidden variable resultant method when based on the Cayley resultant~\cite{Cayley_1848_01} 
finds the solutions to~\eqref{eq:polynomialSystem} by solving the polynomial eigenvalue problem given by
$R_{Cayley}(x_d)v = 0$, where $R_{Cayley}(x_d)$ is a certain matrix polynomial. 
To define it we follow the exposition in~\cite{Chionh_02_01} and first introduce a related Cayley 
function $f_{Cayley}$.
\begin{definition}[Cayley function]
The Cayley function associated with the polynomials $q_1,\ldots,q_d\in\mathbb{C}_{n}[x_1,\ldots,x_{d-1}]$ is a multivariate 
polynomial in $2d-2$ variables, denoted by $f_{Cayley} = f_{Cayley}(q_1,\ldots,q_d)$, 
and is given by
\begin{equation}
f_{Cayley}  
= {\rm det}
\begin{pmatrix} 
q_1(s_1,s_2,\ldots,s_{d-1}) & \ldots & q_d(s_1,s_2,\ldots,s_{d-1})\\[3pt] 
q_1(t_1,s_2,\ldots,s_{d-1}) & \ldots & q_d(t_1,s_2,\ldots,s_{d-1})\\[3pt] 
\vdots & \ddots & \vdots\\[3pt]
q_1(t_1,t_2,\ldots,t_{d-1}) & \ldots & q_d(t_1,t_2,\ldots,t_{d-1})\\[3pt]  
\end{pmatrix}\Bigg/\prod_{i=1}^{d-1} (s_i - t_i).
\label{eq:dixonResultant}
\end{equation} 
\end{definition}

In two dimensions the Cayley function (also known as the B\'{e}zoutian function~\cite{Nakatsukasa_13_01}) takes the 
more familiar form of 
\[
 f_{Cayley} = \frac{1}{s_1-t_1}\det\begin{pmatrix}q_1(s_1) & q_2(s_1)\\[3pt] q_1(t_1) &q_2(t_1)\end{pmatrix} = \frac{q_1(s_1)q_2(t_1)-q_2(s_1)q_1(t_1)}{s_1-t_1}, 
\]
which is of degree at most $n-1$ in $s_1$ and $t_1$.  By carefully applying 
Laplace's formula for the matrix determinant in~\eqref{eq:dixonResultant}, one can see that $f_{Cayley}$ 
is a polynomial of degree $\tau_k \leq kn-1$ in $s_k$ and $t_{d-k}$ for $1\leq k\leq d-1$.  

Note that $f_{Cayley}$ is not the multidimensional resultant (except when $\tau_k = 0$ for all $k$). Instead, $f_{Cayley}$ is a function that is a convenient way to define the Cayley resultant matrix.  

Let $\{\phi_0,\phi_1,\ldots,\}$ be the selected degree-graded polynomial basis. 
The Cayley resultant matrix depends on the polynomial basis and is 
related to the expansion coefficients of 
$f_{Cayley}$ in a tensor-product basis of 
$\{\phi_0,\phi_1,\ldots,\}$. That is, let
\begin{equation} 
 f_{Cayley} = \sum_{i_1=0}^{\tau_1}\!\cdots\!\sum_{i_{d-1}=0}^{\tau_{d-1}}\sum_{j_1=0}^{\tau_{d-1}}\!\cdots\!\sum_{j_{d-1}=0}^{\tau_1} A_{i_1,\ldots,i_{d-1},j_1,\ldots,j_{d-1}}\prod_{k=1}^{d-1}\phi_{i_k}(s_k) \prod_{k=1}^{d-1}\phi_{j_k}(t_k)
\label{eq:expansion} 
\end{equation}
be the tensor-product expansion of the polynomial $f_{Cayley}$, where 
$A$ is a tensor of expansion coefficients of size $(\tau_1+1)\times \cdots \times (\tau_{d-1}+1)\times (\tau_{d-1}+1) \times\cdots \times (\tau_1+1)$. 
The {\em Cayley resultant matrix} is the following unfolding (or matricization) of $A$~\cite[Sec.~2.3]{Ragnarsson_12_01}: 
\begin{definition}[Cayley resultant matrix]
The Cayley resultant matrix associated with $q_1,\ldots,q_d\in\mathbb{C}_n[x_1,\ldots,x_{d-1}]$ with respect to the basis $\{\phi_0,\phi_1,\ldots,\}$ is 
denoted by $R_{Cayley}$ and is the $\left(\prod_{k=1}^{d-1} (\tau_k+1)\right)\times\left(\prod_{k=1}^{d-1} (\tau_k+1)\right)$
matrix formed by the unfolding of the tensor $A$ in~\eqref{eq:expansion}. This unfolding is often denoted
by $A_{\textbf{r}\times\textbf{c}}$, where $\textbf{r} = \{1,\ldots,d-1\}$ and $\textbf{c}=\{d,\ldots,2d-2\}${\em ~\cite[Sec.~2.3]{Ragnarsson_12_01}}.
\end{definition}

For example, when $\tau_k = kn-1$ for $1\leq k\leq d-1$ we have 
for $0\leq i_k,j_{d-k}\leq kn-1$
\[
R_{Cayley}\left(\sum_{k=1}^{d-1} (k-1)!i_kn^{k-1}, \sum_{k=1}^{d-1} j_{d-k}\frac{(d-1)!}{(d-k)!}n^{k-1}\right) = A_{i_1,\ldots,i_{d-1},j_1,\ldots,j_{d-1}}.
\]
This is equivalent to 
{\tt N = factorial(d-1)*n\string^(d-1); R = reshape(A, N, N);} 
in MATLAB, except here the indexing of the matrix 
$R_{Cayley}$ starts at $0$.

For rootfinding, we set 
$q_1 = p_1[x_d],\ldots,q_d=p_d[x_d]$ (thinking of $x_d$ as the ``hidden'' variable). 
Then, $R_{Cayley} = R_{Cayley}(x_d)$ is a square matrix polynomial (see Section~\ref{sec:matrixpoly}). 
If all the polynomials are of maximal degree $n$, then $R_{Cayley}$ is of size 
$(d-1)!n^{d-1}$ and of degree at most $dn$.  
The fact that $(d-1)!n^{d-1} \times dn = d!n^d$ is the maximum number of 
possible solutions that~\eqref{eq:polynomialSystem} can possess (see Lemma~\ref{lem:MaxSolutions}) is a consequence of $R_{Cayley}$
being a resultant matrix. In particular, the eigenvalues of $R_{Cayley}(x_d)$ 
are the $d$th components of the solutions to~\eqref{eq:polynomialSystem} 
and the remaining $d-1$ components of the solutions can in principle be obtained 
from the eigenvectors.

It turns out that evaluating $f_{Cayley}$ at $t_1^{\ast},\ldots,t_{d-1}^{\ast}$
is equivalent to a matrix-vector product with $R_{Cayley}$. This relationship between $R_{Cayley}$ and $f_{Cayley}$ will be 
essential in Section~\ref{subsec:CayleyEigenvectorStructure} for understanding
the eigenvectors of $R_{Cayley}$. 
\begin{lemma} 
Let $d\geq 2$, $\underline{t}^{\ast}\in\mathbb{C}^{d-1}$, and $f_{Cayley}$ and $R_{Cayley}$
be the Cayley function and matrix associated with 
$q_1,\ldots,q_d\in\mathbb{C}_{n}[x_1,\ldots,x_{d-1}]$, respectively. 
If $V$ is the tensor satisfying   
$V_{j_1,\ldots,j_{d-1}} = \prod_{k=1}^{d-1} \phi_{j_k}(t_k^{\ast})$ for $0\leq j_{d-k} \leq \tau_k$, then
we have
\[
\begin{aligned} 
  R_{Cayley} {\rm vec}(V) &= {\rm vec}(Y),
\end{aligned} 
\]
where $Y$ is the tensor that satisfies
\[
 f_{Cayley}(s_1,\ldots,s_{d-1},t_1^{\ast},\ldots,t_{d-1}^{\ast}) = \sum_{i_1=0}^{\tau_1}\cdots\sum_{i_{d-1}=0}^{\tau_{d-1}}Y_{i_1,\ldots,i_{d-1}}\prod_{k=1}^{d-1}\phi_{i_k}(s_k).
\]
\label{lem:rightCayleyVectors}
\end{lemma}
\begin{proof} 
The matrix-vector product $R_{Cayley} {\rm vec}(V) = {\rm vec}(Y)$ is equivalent to the 
following sums: 
\[
\sum_{j_1=0}^{\tau_{d-1}}\cdots\sum_{j_{d-1}=0}^{\tau_1} A_{i_1,\ldots,i_{d-1},j_1,\ldots,j_{d-1}}\prod_{k=1}^{d-1}\phi_{j_k}(t_k^{\ast}) = Y_{i_1,\ldots,i_{d-1}}
\]
for some tensor $Y$. The result follows from~\eqref{eq:expansion}. 
\end{proof}

\subsection{The Cayley resultant as a generalization of Cramer's rule}
In this section we show that for systems of linear 
polynomials, i.e., of total degree $1$, the Cayley 
resultant is precisely Cramer's rule. We believe this connection is folklore, 
but we have been unable to find an existing 
reference that provides a rigorous justification. It gives a first hint 
that the hidden variable resultant method in full generality may be 
numerically unstable.
\begin{theorem}
Let $A$ be a matrix of size $d\times d$, $\underline{x} = (x_1,\ldots,x_d)^T$, and 
$\underline{b}$ a vector of size $d\times 1$. Then, solving the linear polynomial system 
$A \underline{x} + \underline{b} = 0$ by the hidden variable resultant method based on the 
Cayley resultant is equivalent to Cramer's rule for calculating $x_d$.
\label{thm:CayleyCramer}
\end{theorem}
\begin{proof}
Let $A_d$ be the last column of $A$ and $B = A - A_d e_d^T + \underline{b} e_d^T$, where $e_d$ is 
the $d$th canonical vector. Recall that Cramer's rule computes the entry $x_d$ 
in $A \underline{x} = -\underline{b}$ via the formula
$x_d=-\det(B)/\det(A)$. We will show that for the linear polynomial system 
$A \underline{x} + \underline{b} = 0$ we have $f_{Cayley} = \det(B) + x_d\det(A)$. Observe that this, in particular, implies that (since $f_{Cayley}$ has degree $0$ in $s_i,t_i$ for all $i$) $f_{Cayley}=R_{Cayley}=\mathcal{R}_{Cayley}$. Hence, 
the equivalence between Cramer's rule and rootfinding based on the Cayley resultant.

First, using~\eqref{eq:dixonResultant}, we write 
$f_{Cayley} = \det(M)/\det(V)$ where the matrices $M$ and $V$ are
\[ 
V=\begin{bmatrix}
s_1 & t_1  & t_1 & \dots & t_1\\[3pt]
s_2 & s_2 & t_2 & \dots & t_2\\[3pt]
\vdots & \vdots & \vdots & \ddots & \vdots\\[3pt]
s_{d-1} & s_{d-1} & s_{d-1} & \dots & t_{d-1} \\[3pt]
1 & 1 & 1 & \dots & 1
\end{bmatrix}, \qquad   M = B V + x_d A_d e^T,
\]
where $e$ is the $d\times 1$ vector of all ones. (It can be shown by 
induction on $d$ that $\det(V) = \prod_{i=1}^{d-1} (s_i - t_i)$, as required.) 
Using the matrix determinant lemma, we have
\[ 
\det(M) = \det(B) \det(V) + x_d e^T \adj(BV) A_d,
\]
where $\adj(BV)$ is the algebraic adjugate matrix of $BV$. 
Now, recall that $\adj(BV)=\adj(V)\adj(B)$ and observe that
$e^T \adj(V) = \det(V) e_d^T$. Hence, we obtain
\[ 
\frac{\det(M)}{\det(V)} = \det(B)  + x_d (e_d^T \adj(B) A_d). 
\]
Using $e_d^T \adj(B)  \underline{b}= \det(B)$ and the matrix determinant lemma 
one more time, we conclude that
\[ 
\det(A) = \det(B) + e_d^T \adj(B) A_d - e_d^T \adj(B) \underline{b} = e_d^T \adj(B) A_d.
\]
Thus, $f_{Cayley} = \det(B) + x_d\det(A)$ and the resultant method calculates 
$x_d$ via Cramer's formula. 
\end{proof}
 
It is well-known in the literature that Cramer's rule is a numerically unstable
algorithm for solving $A\underline{x} = \underline{b}$~\cite[Sec.~1.10.1]{Higham_02_01}. 
Thus, Theorem~\ref{thm:CayleyCramer} casts significant suspicion on 
the numerical properties of the hidden variable resultant method based on the 
Cayley resultant. 

\subsection{The eigenvector structure of the Cayley resultant matrix}\label{subsec:CayleyEigenvectorStructure}
Ultimately, we wish to use Lemma~\ref{lem:conditionExpression} to estimate
the condition number of the eigenvalues of the Cayley resultant matrix.
To do this we need to know the left and right eigenvectors of $R_{Cayley}$. 
The following lemma shows that the eigenvectors of $R_{Cayley}$ are in 
Vandermonde form\footnote{In one dimension we say that an $N\times 1$ vector $v$ is in Vandermonde form if there is an $x\in\mathbb{C}$ 
such that $v_i = \phi_i(x)$ for $0\leq i\leq N-1$. In higher dimensions, the vector 
${\rm vec}(A)$ is 
in Vandermonde form if $A_{i_1,\ldots,i_d} = \prod_{k=1}^d \phi_{i_k}(x_k)$ for some $x_1,\ldots,x_d\in\mathbb{C}$.}.  To show this we exploit the convenient relationship between 
evaluation of $f_{Cayley}$ and matrix-vector products with $R_{Cayley}$. 
\begin{lemma} 
 Suppose that $\underline{x}^{\ast}=(x^{\ast}_1,\ldots,x^{\ast}_d)\in\mathbb{C}^d$ is a simple root of~\eqref{eq:polynomialSystem}. 
 Let $V$ and $W$ be tensors of size $(\tau_{d-1}+1) \times\cdots \times (\tau_1+1)$ and 
 $(\tau_1+1)\times\cdots \times (\tau_{d-1}+1)$, respectively, defined by 
  \[
  V_{j_1,\ldots,j_{d-1}} = \prod_{k=1}^{d-1} \phi_{j_k}(x_k^{\ast}), \qquad 0\leq j_{k} \leq \tau_{d-k}
 \] 
 and
   \[
  W_{i_1,\ldots,i_{d-1}} = \prod_{k=1}^{d-1} \phi_{i_k}(x_k^{\ast}), \qquad 0\leq i_{k} \leq \tau_k.
 \]
 Then, the vectors ${\rm vec}(V)$ and ${\rm vec}(W)$ are the right and left 
 eigenvectors of the matrix $R_{Cayley}(p_1[x_d^{\ast}],\ldots,p_d[x_d^{\ast}])$ 
 that correspond to the eigenvalue $x_d^{\ast}$.
 \label{lem:CayleyEigenvectorStructure}
\end{lemma}
\begin{proof}
 Let $f_{Cayley} = f_{Cayley} (p_1[x_d^{\ast}],\ldots,p_d[x_d^{\ast}])$ be the 
 Cayley function associated with $p_1[x_d^{\ast}],\ldots,p_d[x_d^{\ast}]$. 
 From~\eqref{eq:dixonResultant} we find that $f_{Cayley}(s_1,\ldots,s_{d-1},x_1^{\ast},\ldots,x_{d-1}^{\ast})=0$
 because the determinant of a matrix with a vanishing last row is zero. Moreover, 
 by Lemma~\ref{lem:rightCayleyVectors} we have
 \[
 0=f_{Cayley}(s_1,\ldots,s_{d-1},x_1^{\ast},\ldots,x_{d-1}^{\ast})= \sum_{i_1=0}^{\tau_1}\cdots\sum_{i_{d-1}=0}^{\tau_{d-1}} Y_{i_1,\ldots,i_{d-1}}\prod_{k=1}^{d-1} \phi_{i_k}(s_k).  
 \]
 Since $\{\phi_0,\phi_1,\ldots,\}$ is a polynomial basis we must conclude that 
 $Y = 0$, and hence, $R_{Cayley}(x_d^{\ast}) v = 0$ with $v = {\rm vec}(V)$. 
 In other words, $v$ is a right eigenvector of $R_{Cayley}$ corresponding 
 to the eigenvalue $x_d^{\ast}$ (see Definition~\ref{def:MatpolyEigenvector}). 
 
 An analogous derivation shows that ${\rm vec}(W)$ is a 
 left eigenvector of $R_{Cayley}$. 
\end{proof}

\subsection{On the generalized Rayleigh quotient of the Cayley resultant matrix}
To bound $\kappa(x_d^{\ast},R_{Cayley})$ we need to 
bound the absolute value of the generalized Rayleigh quotient of $R_{Cayley}'(x_d)$ (see~Lemma~\ref{lem:conditionExpression}), 
whenever $\underline{x}^{\ast}\in\mathbb{C}^d$ is such that $x_d^{\ast}$ is a 
simple eigenvalue of $R_{Cayley}(x_d)$, i.e., there are no other solutions to~\eqref{eq:polynomialSystem} with the same $d$th component.  
In a similar style to the proof of Lemma~\ref{lem:CayleyEigenvectorStructure} we show
this by exploiting the relation between evaluating the derivative of 
$f_{Cayley}$ and matrix-vector products with $R_{Cayley}'(x_d)$.
\begin{theorem} 
Let $p_1,\dots,p_d$ be the polynomials in~\eqref{eq:polynomialSystem}, 
$\underline{x}^{\ast}\in\mathbb{C}^{d}$ a solution 
of~\eqref{eq:polynomialSystem}, and $f_{Cayley}(x_d)$ the Cayley function
associated with $q_1=p_1[x_d],\dots,q_d=p_d[x_d]$. We have
\[
f_{Cayley}'(x_d^{\ast}) \Big|_{\genfrac{}{}{0pt}{}{s_k = t_k = x_k^{\ast}}{1\leq k\leq d-1}} = {\rm det}(J(x_d^{\ast})),
\]
where $J(\underline{x}^{\ast})$ is the Jacobian matrix in~\eqref{eq:Jacobian}. That is, 
$f_{Cayley}'(x_d^{\ast})$ evaluated at $s_k = t_k = x_k^{\ast}$ for $1\leq k\leq d-1$ is 
equal to the determinant of the Jacobian.
\label{thm:jacobianExpression}
\end{theorem}
\begin{proof}
 Recall from~\eqref{eq:dixonResultant} that $f_{Cayley}(x_d)$ is a 
 polynomial in $s_1,\ldots,s_{d-1}$ and $t_1,\ldots,t_{d-1}$ written 
 in terms of a matrix determinant, and set $q_1=p_1[x_d],\ldots,q_d=p_d[x_d]$. 
 The determinant in~\eqref{eq:dixonResultant} for 
 $f_{Cayley}(x_d)$ can be expanded to obtain 
 \[
  f_{Cayley}(x_d) = \frac{1}{\prod_{i=1}^{d-1} (s_i-t_i)}\sum_{\sigma \in S_d} (-1)^\sigma \prod_{i=1}^d p_{\sigma_i}[x_d](t_1,\dots,t_{i-1},s_i,\dots,s_{d-1}),
 \] 
where $S_d$ is the symmetric group of $\{1,\ldots,d\}$ and $(-1)^\sigma$ is the 
signature of the permutation $\sigma$. When we evaluate $f_{Cayley}(x_d)$ at $s_k = t_k = x_k^{\ast}$ for $1\leq k\leq d-1$ 
the denominator vanishes, and hence, so does the numerator because $f_{Cayley}(x_d)$ is a polynomial. 
Thus, by L'Hospital's rule, $f_{Cayley}'(x_d^{\ast})$ evaluated $s_k = t_k = x_k^{\ast}$ 
for $1\leq k\leq d-1$ is equal to 
\begin{equation}
\frac{\partial^{d}}{\partial s_1\cdots\partial s_{d-1}\partial x_{d}}\sum_{\sigma \in S_d} (-1)^\sigma \prod_{i=1}^d p_{\sigma_i}[x_d](t_1,\dots,t_{i-1},s_i,\dots,s_{d-1})
\label{eq:diffDet}
\end{equation}
evaluated at $s_k = x_k^{\ast}$, $t_k = x_k^{\ast}$, and $x_d = x_d^{\ast}$.
In principle, one could now apply the product rule and evaluate the combinatorially many terms in~\eqref{eq:diffDet}. Instead, 
we note that after applying the product rule a term is zero if it contains $p_{\sigma_i}(\underline{x}^{\ast})$ 
for any $\sigma\in S_d$ and $1\leq i\leq d$ (since $\underline{x}^{\ast}$ is a solution to~\eqref{eq:polynomialSystem}). 
There are precisely $d$ partial derivatives and $d$ terms in each product so that any nonzero term 
when expanding~\ref{eq:diffDet} has each $p_k$ differentiated precisely once. Finally, note that for each 
$1\leq k\leq d-1$ only the $1\leq i\leq k$ terms in the product depend 
on $s_k$. Hence, from~\eqref{eq:diffDet} we obtain
\[
 f_{Cayley}'(x_d^{\ast}) \bigg|_{\genfrac{}{}{0pt}{}{s_k = t_k = x_k^{\ast}}{1\leq k\leq d-1}} = \sum_{\sigma \in S_d} (-1)^\sigma \prod_{i=1}^d \frac{\partial p_{\sigma_i}}{\partial x_i}(\underline{x}^{\ast}).
\]
The result follows because the last expression is the determinant of the Jacobian 
matrix evaluated at $\underline{x}^{\ast}$. 
\end{proof}

As a consequence of Theorem~\ref{thm:jacobianExpression} we have the following
unavoidable conclusion that mathematically explains the numerical difficulties 
that practitioners 
have been experiencing with hidden variable resultant methods based on the Cayley 
resultant.
 \begin{theorem} 
Let $d\geq 2$. Then, there exist $p_1,\dots,p_d$ in~\eqref{eq:polynomialSystem} 
with a 
  simple root $\underline{x}^{\ast}\in\mathbb{C}^{d}$ such that 
   \[
  \kappa(x_d^{\ast},R_{Cayley}) \geq \|J(\underline{x}^{\ast})^{-1}\|_2^{d}
 \]
 and $\|J(\underline{x}^{\ast})^{-1}\|_2>1$.
 Thus, an eigenvalue of $R_{Cayley}(x_d)$ can be more sensitive 
 to perturbations than the corresponding root by a factor that grows 
 exponentially with $d$.
 \label{thm:condTheorem}
 \end{theorem}
\begin{proof} 
 Using Lemma~\ref{lem:rightCayleyVectors}, Theorem~\ref{thm:jacobianExpression} has the following 
 equivalent matrix form: 
 \[
  w^T R_{Cayley}'(x_d^{\ast}) v = \det(J(\underline{x}^{\ast})), 
 \]
 where $v = {\rm vec}(V)$, $w = {\rm vec}(W)$, and $V$ and $W$ are given in Lemma~\ref{lem:CayleyEigenvectorStructure}.
 Since $\phi_0 = 1$, we know that $\|v\|_2 \geq 1$ and $\|w\|_2 \geq 1$. Hence, by Lemma~\ref{lem:conditionExpression}
 \[
  \kappa(x_d^{\ast},R_{Cayley}) \geq |\det (J(\underline{x}^{\ast}))|^{-1}.
 \]
Denoting the singular values~\cite[Sec.~7.3]{Horn_13_01} of the matrix $J(\underline{x}^{\ast})$ by $\sigma_i$ , 
select $p_1,\dots,p_d$ and $\underline{x}^{\ast}\in\mathbb{C}^{d}$ such 
that $\left|{\rm det} (J(\underline{x}^{\ast})) \right|= \prod_{i=1}^d \sigma_i = \sigma_d^{d}$. 
Such polynomial systems do exist, for example, linear polynomial systems where $M \underline{x} - M \underline{x}^{\ast}=0$ 
and $M$ is a matrix with singular values $\sigma_1 = \sigma_2 = \cdots = \sigma_d$. To ensure that 
$\|J(\underline{x}^{\ast})^{-1}\|_2>1$ we also require $\sigma_d<1$.
Then, we have
\[
 \kappa(x_d^{\ast},R_{Cayley})^{-1} \leq \left|{\rm det} (J(\underline{x}^{\ast})) \right|= \prod_{i=1}^d \sigma_i = \sigma_d^{d} = \|J(\underline{x}^{\ast})^{-1}\|_2^{-d}.
\]
The result follows. 
\end{proof}

\begin{example}\label{exC}
Let $Q$ be a $d \times d$ orthogonal matrix, $Q Q^T = I_d$, having elements $q_{ij}$ for $i,j=1,\dots,d$, and let $\sigma < 1$. Consider the system of polynomial equations
\[ p_i = x_i^2 + \sigma \sum_{j=1}^d q_{ij} x_j = 0, \ \ \ i=1,\dots,d.  \]
The origin, $x^\ast=0 \in \mathbb{C}^d$, is a simple root of this system of equations. The Jacobian of the system at $0$ is  $J=\sigma Q$, and hence, the absolute conditioning of the problem is $\| J^{-1}\| = \sigma^{-1}$. Constructing the Cayley resultant matrix polynomial in the monomial basis, one readily sees that for this example the right and left eigenvectors for the eigenvalue $x_d^\ast=0$ satisfy $\|v\|=\|w\|=1$. As a consequence, $\kappa(x_d^{\ast},R_{Cayley})=\sigma^{-d}$.
\end{example}

We emphasize that this numerical instability is truly spectacular, affects the accuracy of $x_d^{\ast}$, 
and can grow exponentially with the dimension $d$. 

Moreover, Theorem~\ref{thm:condTheorem} 
holds for any degree-graded polynomial basis selected to represent $p_1,\ldots,p_d$ as long as $\phi_0 = 1$.  In particular, the associated numerical instability 
cannot be resolved in general by a special choice of polynomial basis. 
 
Theorem~\ref{thm:condTheorem} is pessimistic and importantly does not imply that the 
resultant method always loses accuracy, just that it might. In general, one must 
know the solutions to~\eqref{eq:polynomialSystem} and the singular values of the Jacobian 
matrix to be able to predict if and when the resultant method will be accurate. 

One should note that Theorem~\ref{thm:condTheorem} concerns absolute conditioning
and one may may wonder if a similar phenomenon also occurs in the relative sense.
In Section~\ref{sec:absRel} we show that the relative conditioning can also be increased by an 
exponential factor with $d$. 

\section{The Sylvester matrix is numerically unstable for bivariate rootfinding}\label{sec:Sylvester}
A popular alternative in two dimensions to the Cayley resultant matrix is the 
Sylvester matrix~\cite[Chap.~3]{Cox_13_01}, denoted here by $R_{Sylv}$. We now set out to show 
that the hidden variable 
resultant  based on $R_{Sylv}$ is also numerically unstable. However, since $d=2$ the instability has only 
a moderate impact in practice as the conditioning can only be at most squared.
With care, practical bivariate rootfinders can be based on the Sylvester resultant~\cite{Sorber_14_01} 
though there is the possibility that a handful digits are lost.

A neat way to define 
the Sylvester matrix that accommodates nonmonomial polynomial bases
is to define the matrix one row at a time. 

\begin{definition}[Sylvester matrix] 
 Let $q_1$ and $q_2$ be two univariate polynomials in $\mathbb{C}_{n}[x_1]$ of degree exactly $\tau_1$ and $\tau_2$, respectively.
Then, the Sylvester matrix $R_{Sylv}\in\mathbb{C}^{(\tau_1+\tau_2) \times (\tau_1+\tau_2)}$ 
associated with $q_1$ and $q_2$ is defined row-by-row as
\[
 R_{Sylv}\left( i , \, :\, \right) = Y^{i,1},  \qquad 0\leq i\leq \tau_{2}-1,
\]
where $Y^{i,1}$ is the row vector of coefficients such that $q_1(x)\phi_{i}(x) = \sum_{k=0}^{\tau_1+\tau_2-1} Y_k^{i,1} \phi_{k}(x)$
and 
\[
 R_{Sylv}\left( i + \tau_2 , \, :\, \right) = Y^{i,2},\quad 0\leq i\leq \tau_{1}-1,
\]
where $Y^{i,2}$ is the row vector of coefficients such that $q_2(x)\phi_{i}(x) = \sum_{k=0}^{\tau_1+\tau_2-1} Y_k^{i,2} \phi_{k}(x)$. 
\label{def:SylvesterMatrix} 
\end{definition}

In the monomial basis, i.e., $\phi_k(x) = x^k$, Definition~\ref{def:SylvesterMatrix} gives the 
Sylvester\footnote{Variants of~\eqref{eq:Monomial} include its transpose or a permutation of 
its rows and/or columns. Our analysis still applies after 
these aesthetic modifications with an appropriate change of 
indices. We have selected this variant for the convenience of 
indexing notation.} matrix of size $(\tau_1+\tau_2)\times(\tau_1+\tau_2)$ as~\cite[Chap.~3]{Cox_13_01}: 
\begin{equation}
R_{Sylv} = 
\begin{pmatrix} 
a_{0} & a_{1} & \ldots & a_{\tau_1} &  &  \\[3pt]
  & \ddots & \ddots & \ddots & \ddots \\[3pt]
& & a_{0} & a_{1} &\ldots & a_{\tau_1} \\[3pt] 
b_{0} & b_{1} & \ldots & b_{\tau_2} &  & \\[3pt]
  &  \ddots & \ddots & \ddots & \ddots \\[3pt]
& & b_{0} & b_{1} &\ldots & b_{\tau_2} \\[3pt]
\end{pmatrix}\begin{matrix} 
\coolrightbrace{x\\[3pt]y\vphantom{\ddots}\\[3pt]y\\[3pt]}{\tau_2 \text{ rows}}\\
\coolrightbrace{x\\[3pt]y\vphantom{\ddots}\\[3pt]y\\[3pt]}{\tau_1 \text{ rows}}\\
\end{matrix}
\label{eq:Monomial}
\end{equation}
where $q_1(x) = \sum_{k=0}^{\tau_1} a_k x^k$ and $q_2(x) = \sum_{k=0}^{\tau_2} b_kx^k$. 

\subsection{A generalization of Clenshaw's algorithm for degree-graded polynomial bases}\label{subsec:GeneralizedClenshaw}
Our goal is to use Lemma~\ref{lem:conditionExpression} to bound 
the condition number of the eigenvalues of the Sylvester matrix.
It turns out the right eigenvectors of $R_{Sylv}$ are in Vandermonde form. However, 
the left eigenvectors have a more peculiar structure and are related to 
the byproducts of a generalized Clenshaw's algorithm for degree-graded polynomial 
bases (see Lemma~\ref{lem:SylvesterEigenvectorStructure}). We develop a Clenshaw's
algorithm for degree-graded bases in this section with derivations of its 
properties in Appendix~\ref{sec:appendix}. 

The selected polynomial basis $\phi_0,\phi_1,\ldots,$ is degree-graded 
and hence, satisfies a recurrence relation of the form 
\begin{equation}
 \phi_{k+1}(x) = (\alpha_kx + \beta_k)\phi_k(x) + \sum_{j=1}^{k} \gamma_{k,j}\phi_{j-1}(x), \qquad k\geq 1,
\label{eq:degreeGradedRecurrence}
\end{equation} 
where $\phi_{1}(x) = (\alpha_0x + \beta_0)\phi_0(x)$ and $\phi_0(x) = 1$. 
If $\phi_0,\phi_1,\ldots,$ is an orthogonal polynomial basis, 
then~\eqref{eq:degreeGradedRecurrence} is a three-term recurrence and 
it is standard to employ Clenshaw's algorithm~\cite{Clenshaw_55_01} to evaluate 
polynomials expressed as $p(x) = \sum_{k=0}^n a_k\phi_k(x)$. This procedure can 
be extended to any degree-graded polynomial basis. 

Let $p(x)$ be expressed as $p(x) = \sum_{k=0}^n a_k\phi_k(x)$, where $\phi_0,\ldots,\phi_n$
is a degree-graded polynomial basis. 
One can evaluate $p(x)$ via the following procedure: Let $b_{n+1}[p](x) = 0$, and calculate 
$b_n[p](x),\ldots,b_1[p](x)$ from the following recurrence relation:
\begin{equation} 
  b_k[p](x) = a_k + (\alpha_kx + \beta_k)b_{k+1}[p](x) + \sum_{j=k+1}^{n-1} \gamma_{j,k+1}b_{j+1}[p](x), \qquad 1\leq k\leq n.
\label{eq:ClenshawLikeAlg}
\end{equation} 
We refer to the quantities $b_1[p](x),\ldots,b_{n+1}[p](x)$ as {\em Clenshaw shifts}
(in the monomial case they are called Horner shifts~\cite{Teran_10_01}). 
The value $p(x)$ can be written in terms of the Clenshaw shifts\footnote{Note that, although Lemma~\ref{lem:ClenshawDegreeGraded} is stated in a general form and holds for \emph{any} degree-graded basis, in this paper we fix the normalization $\max_{x \in \Omega} |\phi_j(x)|=1$, that implies in particular $\phi_0=1$ simplifying \eqref {lem:ClenshawDegreeGraded}.}.
\begin{lemma} 
Let $n$ be a positive integer, $x\in\mathbb{C}$, $\phi_0,\ldots,\phi_n$ a degree-graded basis satisfying~\eqref{eq:degreeGradedRecurrence}, 
$p(x) = \sum_{k=0}^n a_k\phi_k(x)$, and $b_{n+1}[p](x),\ldots,b_1[p](x)$ the Clenshaw shifts satisfying~\eqref{eq:ClenshawLikeAlg}. Then, 
 \begin{equation}
  p(x) = a_0\phi_0(x) + \phi_1(x)b_1[p](x) + \sum_{i=1}^{n-1} \gamma_{i,1}b_{i+1}[p](x).
 \label{eq:ClenshawDegreeGraded}
 \end{equation}
 \label{lem:ClenshawDegreeGraded}
\end{lemma}
\begin{proof} 
 See Appendix~\ref{sec:appendix}.
\end{proof}

Clenshaw's algorithm for degree-graded polynomial bases is summarized in Figure~\ref{fig:ClenshawAlgorithm}. 
We note that because of the full recurrence in~\eqref{eq:ClenshawLikeAlg} the 
algorithm requires $\mathcal{O}(n^2)$ operations to evaluate $p(x)$. Though this 
algorithm may not be of significant practical importance, it is of theoretical interest
for the conditioning analysis of some linearizations from the so-called $\mathbb{L}_1$- 
or $\mathbb{L}_2$-spaces~\cite{Mackey_06_01} when degree-graded bases are employed~\cite{Nakatsukasa_13_01}. 

\begin{figure} 
 \centering 
 \fbox{
 \parbox{.8\textwidth}{ 
 \textbf{Clenshaw's algorithm for degree-graded polynomial bases }
 
 \vspace{.2cm}
 
 Let $\phi_0,\phi_1,\ldots,$ satisfy~\eqref{eq:degreeGradedRecurrence} and $p(x) = \sum_{k=0}^n a_k\phi_k(x)$.
 
 \vspace{.1cm}
 
 Set $b_{n+1}[p](x) = 0$. 
 
  \vspace{.1cm}
 
 $\quad$\textbf{for} $k = n,n-1,\ldots,1$ \textbf{do} 
 
 $\qquad$ $b_k[p](x) = a_k + (\alpha_kx + \beta_k)b_{k+1}[p](x) + \sum_{j=k+1}^{n-1} \gamma_{j,k+1}b_{j+1}[p](x)$
 
 $\quad$\textbf{end}
 
  \vspace{.1cm}
 
 $p(x) = a_0\phi_0(x) + \phi_1(x)b_1[p](x) + \sum_{j=1}^{n-1} \gamma_{j,1}b_{j+1}[p](x)$.
 }}
 \caption{Clenshaw's algorithm for evaluating polynomials expressed in a 
 degree-graded basis.}
\label{fig:ClenshawAlgorithm}
\end{figure}

There is a remarkable and interesting connection between Clenshaw shifts and 
the quotient $(p(x)-p(y))/(x-y)$, which will be useful when deriving the left
eigenvectors of $R_{Sylv}$. 
\begin{theorem} 
With the same set up as Lemma~\ref{lem:ClenshawDegreeGraded} we have 
\begin{equation}
\frac{p(x)-p(y)}{x-y} = \sum_{i=0}^{n-1} \alpha_i b_{i+1}[p](y)\phi_i(x), \quad x\neq y
\label{eq:diffDegreeGraded}
\end{equation}
and
\begin{equation}
 p'(x) = \sum_{i=0}^{n-1} \alpha_i b_{i+1}[p](x)\phi_i(x).
\label{eq:diffClenshaw}
\end{equation}
\label{thm:diffDegreeGraded}
\end{theorem} 
\begin{proof} 
 See Appendix~\ref{sec:appendix}.
\end{proof}

The relation between the derivative and Clenshaw shifts in~\eqref{eq:diffClenshaw} has been 
noted by Skrzipek for orthogonal polynomial bases in~\cite{Skrzipek_98_01}, where it was
used to construct a so-called {\em extended Clenshaw's algorithm} for evaluating polynomial derivatives.
Using Theorem~\ref{thm:diffDegreeGraded} and~\cite{Skrzipek_98_01} an extended Clenshaw's algorithm 
for polynomials expressed in a degree-graded basis is immediate. 

\subsection{The eigenvector structure of the Sylvester matrix} 
We now set $q_1 = p_1[x_2]$ and $q_2=p_2[x_2]$ (considering $x_2$ as the 
hidden variable), 
and we are interested in the eigenvectors of the matrix polynomial $R_{Sylv}(x_2^{\ast})$,
when $(x_1^{\ast},x_2^{\ast})$ is a solution to~\eqref{eq:polynomialSystem} when $d = 2$. 
It turns out that the right eigenvectors of $R_{Sylv}(x_2^{\ast})$ are in 
Vandermonde form, while the left eigenvectors are related to the Clenshaw shifts (see 
Section~\ref{subsec:GeneralizedClenshaw}).
\begin{lemma} 
 Suppose that $\underline{x}^{\ast}=(x_1^{\ast},x_2^{\ast})$ is a simple root 
 of~\eqref{eq:polynomialSystem} and that $p_1[x_2]$ and $p_2[x_2]$ are of 
 degree $\tau_1$ and $\tau_2$, respectively, in $x_1$. 
 The right eigenvector of $R_{Sylv}(x_2^{\ast})$ corresponding to the eigenvalue 
 $x_2^{\ast}$ is 
 \[
  v_k = \phi_{k}(x_1^{\ast}), \qquad 0\leq k\leq \tau_1+\tau_2-1, 
  \]
  and the left eigenvector is defined as
  \[
  w_i = \begin{cases}
  -\alpha_{i} b_{i+1}[q_{2}](x_1^{\ast}), & 0\leq i\leq \tau_2-1,\\
  \alpha_{i-\tau_2} b_{i-\tau_2+1}[q_{1}](x_1^{\ast}), & \tau_2\leq i\leq \tau_1+\tau_2-1,
  \end{cases}
  \]
 where $q_j = p_j[x_2^{\ast}]$ and $b_k[q_j](x_1^{\ast})$ are the Clenshaw shifts 
 with respect to $\{\phi_0,\phi_1,\ldots,\}$, while the coefficients $\alpha_i$ are defined as in \eqref{eq:degreeGradedRecurrence}. 
 \label{lem:SylvesterEigenvectorStructure}
\end{lemma}
\begin{proof} 
By construction we have, for $0\leq i\leq \tau_2-1$,
 \[
  R_{Sylv}\left( i, \, : \right) v =  \sum_{k=0}^{\tau_1+\tau_2-1} Y_{k}^{i,1}(x_2^{\ast}) \phi_{k}(x_1^{\ast}) = q_1(x_1^{\ast})\phi_i(x_1^{\ast}) = 0
\]
and, for $0\leq i\leq \tau_1-1$,
\[
   R_{Sylv}\left( i+\tau_2, \, : \right) v =  \sum_{k=0}^{\tau_1+\tau_2-1} Y_{k}^{i,2}(x_2^{\ast}) \phi_{k}(x_1^{\ast}) = q_2(x_1^{\ast})\phi_i(x_1^{\ast}) = 0. 
\]
Thus, $v$ is a right eigenvector of $R_{Sylv}(x_2^{\ast})$ corresponding to the 
eigenvalue $x_2^{\ast}$. 

For the left eigenvector, first note that for any vector $\Phi$ of 
the form $\Phi_k = \phi_{k}(x)$ for $0\leq k\leq \tau_1+\tau_2-1$ we have
by Theorem~\ref{thm:diffDegreeGraded}
\[ 
\begin{aligned} 
 w^T R_{Sylv}(x_2^{\ast}) \Phi &= -\sum_{i=0}^{\tau_2-1} \alpha_{i}b_{i+1}[q_2](x_1^{\ast})\phi_{i}(x)q_1(x) + \sum_{i=0}^{\tau_1-1} \alpha_{i}b_{i+1}[q_1](x_1^{\ast})\phi_{i}(x)q_2(x)\\
 &=-\frac{q_2(x) - q_2(x_1^{\ast})}{x-x_1^{\ast}}q_1(x) + \frac{q_1(x) - q_1(x_1^{\ast})}{x-x_1^{\ast}}q_2(x)\\
 &=-\frac{q_2(x)}{x-x_1^{\ast}}q_1(x) + \frac{q_1(x)}{x-x_1^{\ast}}q_2(x) = 0,
\end{aligned}
\]
where the second from last equality follows because $q_1(x_1^{\ast}) = q_2(x_1^{\ast}) = 0$. 
Since~\eqref{thm:diffDegreeGraded} holds for any $x$ and 
$\{\phi_0,\phi_1,\ldots,\phi_{\tau_1+\tau_2-1}\}$ is a 
basis of $\mathbb{C}_{\tau_1+\tau_2-1}[x]$, we deduce that
$w^T R_{Sylv}(x_2^{\ast}) = 0$, and hence, $w$ is a left eigenvector of $R_{Sylv}$ corresponding to the 
eigenvalue $x_2^{\ast}$. 
\end{proof}

\subsection{On the generalized Rayleigh quotient of the Sylvester matrix} 
To bound $\kappa(R_{Sylv},x_d^{\ast})$ we look at the absolute value of the generalized 
Rayleigh quotient of $R'_{Sylv}(x_2^{\ast})$, whenever $\underline{x}^{\ast}$ is such that $x_2^{\ast}$ is a simple eigenvalue of $R_{Sylv}(x_2)$. Lemma~\ref{lem:SylvesterEigenvectorStructure}
allows us to show how the generalized Rayleigh 
quotient of $R_{Sylv}'(x_2^{\ast})$ relates to the determinant of the Jacobian. 
\begin{lemma}
With the same assumptions as in Lemma~\ref{lem:SylvesterEigenvectorStructure}, we have
\[
\frac{|w^TR_{Sylv}'(x_2^{\ast})v|}{\|v\|_2\|w\|_2} \leq \frac{|{\rm det}\left( J(\underline{x}^{\ast})\right)|}{\|w\|_2},
\]
where $w$ and $v$ are the left and right eigenvectors of $R_{Sylv}$, respectively, and 
$J(\underline{x}^{\ast})$ is the Jacobian matrix in~\eqref{eq:Jacobian}. 
\label{lem:SylvesterJacobian}
\end{lemma}
\begin{proof}
By Lemma~\ref{lem:SylvesterEigenvectorStructure} we know the 
structure of $v$ and $w$. Hence, we have 
\[
\begin{aligned}
\!w^TR_{Sylv}'(x_2^{\ast})v &= -\!\!\sum_{i=0}^{\tau_2-1}\!\! \alpha_{i}b_{i+1}[q_2](x_1^{\ast})\phi_{i}(x_1^{\ast})\frac{\partial q_1}{\partial x_2}(x_1^{\ast}) + \!\!\sum_{i=0}^{\tau_1-1}\!\! \alpha_{i}b_{i+1}[q_1](x_1^{\ast})\phi_{i}(x_1^{\ast})\frac{\partial q_2}{\partial x_2}(x_1^{\ast})\\
& = -\frac{\partial q_1}{\partial x_2}(x_1^{\ast})\frac{\partial q_2}{\partial x_1}(x_1^{\ast}) + \frac{\partial q_1}{\partial x_1}(x_1^{\ast})\frac{\partial q_2}{\partial x_2}(x_1^{\ast}),
\end{aligned}
\]
where the last equality used the relation in~\eqref{eq:diffClenshaw}. The result now follows
since this final expression equals ${\rm det}\left( J(\underline{x}^{\ast})\right)$ and 
since $\phi_0 = 1$ we have $\|v\|_2\geq 1$. 
\end{proof}

 \begin{theorem} 
There exist $p_1$ and $p_2$ in~\eqref{eq:polynomialSystem} 
with a simple root $\underline{x}^{\ast}\in\mathbb{C}^{2}$ such that 
   \[
  \kappa(x_2^{\ast},R_{Sylv}) \geq \|J(\underline{x}^{\ast})^{-1}\|_2^{2}
 \]
 and $\|J(\underline{x}^{\ast})^{-1}\|_2> 1$.
 Thus, an eigenvalue of $R_{Sylv}(x_2)$ can be squared more sensitive 
 to perturbations than the corresponding root in the absolute sense.
 \label{thm:condSylvester}
 \end{theorem}
\begin{proof} 
We give an example for which $\|w\|_2\geq 1$ in Lemma~\ref{lem:SylvesterJacobian}. 
For some positive parameter $u$ and for some $n \geq 2$ consider the polynomials
\[ 
p_1(x_1,x_2) = x_1^n x_2^n + u^{1/2} x_1, \ \ \ p_2(x_1,x_2) = \alpha^{-1}_{n-1}(x_1^n + x_2^n)  + u^{1/2} x_2.
\]
One can verify that $\underline{x}^\ast = (0,0)$ is a common root\footnote{By a change of variables, there is an analogous example with a solution anywhere in the complex plane.}. 
Since $| b_n[q_2](0) | = \alpha_{n-1} \alpha_{n-1}^{-1} = 1$ we 
have $\| w \|_2 \geq 1$. The result then follows from 
$|\det(J(\underline{x}^\ast))| = \|J(\underline{x}^{\ast})^{-1}\|_2^{-2}$ and
Lemma~\ref{lem:SylvesterJacobian}.
\end{proof}

\begin{example}\label{exS}
Let us specialize Example \ref{exC} to $d=2$, i.e., for some $\sigma < 1$ and $\alpha^2 + \beta^2 = 1$ let us consider the system
\[ p_1 = x_1^2 + \sigma(\alpha x_1 + \beta x_2) = 0, \ \ \ p_2 = x_2^2 + \sigma (- \beta x_1 + \alpha x_2)= 0.  \]
Again, for the solution $(x_1^\ast,x_2^\ast)=(0,0)$ we have $\| J^{-1} \| = \sigma^{-1}$. Building the Sylvester matrix in the monomial basis, we obtain

\[ R_{Sylv} = \begin{bmatrix}
\sigma \beta x_2 & \sigma \alpha & 1 \\
 x_2^2 + \sigma \alpha x_2 & - \sigma \beta & 0\\
0 & x_2^2 + \sigma \alpha x_2 & -\sigma \beta
\end{bmatrix}.  \]
As predicted by the theory, $x_2^\ast=0$ is an eigenvalue with corresponding right and left eigenvectors, respectively, $v=\begin{bmatrix}
1 & 0 & 0
\end{bmatrix}^T$ and $w=\begin{bmatrix}
\sigma \beta  & \sigma \alpha & 1
\end{bmatrix}^T$. Moreover, it is readily checked that, as expected, $w^T R_{Sylv}'(0) v = \sigma^2$. Therefore,  
\[ \kappa(x_2^{\ast},R_{Sylv}) = \frac{\sqrt{1+\sigma^2}}{\sigma^2} > \sigma^{-2}.  \]

\end{example}

Theorem~\ref{thm:condSylvester} mathematically explains the numerical 
difficulties that practitioners 
have been experiencing with hidden variable resultant methods based on the Sylvester 
resultant. There are successful bivariate rootfinders based on this 
methodology~\cite{Sorber_14_01} for low degree polynomial systems and 
it is a testimony to those authors that 
they have developed algorithmic remedies (not cures) for the inherent 
numerical instability. 

We emphasize that Theorem~\ref{thm:condSylvester} holds for any normalized 
degree-graded polynomial basis. Thus, the mild numerical instability cannot, in 
general, be overcome by working in a different degree-graded polynomial basis.  

The example in the proof of Theorem~\ref{thm:condSylvester} is 
quite alarming for a practitioner since if $u$ is the unit machine roundoff, then 
we have $\| J(0,0)^{-1} \|_2 = u^{-1/2}$ and 
$\kappa(x_2^{\ast},R_{Sylv}) = u^{-1}$. Thus, a numerical rootfinder 
based on the Sylvester matrix may entirely miss a solution that 
has a condition number larger than $u^{-1/2}$. A stable rootfinder should not miss 
such a solution.  

When $d = 2$, we can use Theorem~\ref{thm:condTheorem} and Lemma~\ref{lem:SylvesterJacobian} 
to conclude that the ratio between the conditioning of the Cayley and Sylvester resultant 
matrices for the same eigenvalue $x_2^\ast$ 
is equal to $\| v \|_2/\| w \|_2$, where $v$ and $w$ are the right and left 
eigenvector of $R_{Sylv}(x_2^{\ast})$ associated with the eigenvalue $x_2^\ast$. 
This provides theoretical support for the numerical observations in~\cite{Nakatsukasa_15_01}. However, it seems 
difficult to predict {\em a priori} if the Cayley or Sylvester matrix 
will behave better numerically. For real polynomials and $d=2$, the 
Cayley resultant matrix is symmetric and this structure can be exploited~\cite{Nakatsukasa_15_01}. 
In the monomial basis, the Sylvester matrix is two stacked Toeplitz 
matrices (see~\eqref{eq:Monomial}). It may be that structural 
differences like these are more important than their relatively similar numerical 
properties when $d = 2$. 

\section{A discussion on relative and absolute conditioning}\label{sec:absRel}
Let $X(D)$ be the solution of a mathematical problem depending on data $D$. In general, with the very mild assumption that $D$ and $X$ lie in Banach spaces, it is possible to define the absolute condition number of the problem by perturbing the data to $D + \delta D$ and studying the behaviour of the perturbed solution $\hat X(D+\delta D) = X(D) + \delta X(D,\delta D)$:
 \[ \cabs = \lim_{\epsilon \rightarrow 0} \sup_{\| \delta D \| \leq \epsilon} \frac{ \| \delta X \|}{\| \delta D \|}.   \]
 Similarly, a relative condition number can be defined by looking at the limit ratios of \emph{relative} changes.
 \[ \crel = \lim_{\epsilon \rightarrow 0} \sup_{\| \delta D \| \leq \epsilon \| D \|} \frac{ \| \delta X \|}{\| \delta D \|} \frac{\| D \|}{\| X \|} = \cabs \frac{\|D\|}{\| X \|}.   \]

In this paper, we have compared two absolute condition numbers. One is given by Proposition~ \ref{def:conditioning}: there, $X=\underline{x}^*$ is a solution of \eqref{eq:polynomialSystem} while $D=(p_1,\dots,p_d)$ is the set of polynomials in~\eqref{eq:polynomialSystem}. The other is given by Lemma~\ref{lem:conditionExpression}, where $D$ is a matrix polynomial and $X=x_d^*$ is the $d$th component of $\underline{x}^*$.

To quote N.~J.~Higham~\cite[p.~56]{Higham_08_01}: ``Usually, it is 
the relative condition number that is of interest, but it is more convenient to state results 
for the absolute condition number''. This remark applies to our analysis as well. We have found it convenient to study the absolute condition number, but when attempting to solve the rootfinding problem in floating point arithmetic it is natural to allow for relatively small perturbations, and thus to study the relative condition number. Hence, a natural question is whether the exponential increase of the absolute condition number 
in Theorem~\ref{thm:condTheorem} and the squaring in Theorem~\ref{thm:condSylvester} 
causes a similar effect in the relative condition number.  


It is not immediate that the exponential increase of the absolute condition number 
leads to the same effect in the relative sense. We have found examples where 
the exponential increase of the absolute condition number is perfectly counterbalanced 
by an exponentially small Cayley resultant matrix. For instance, linear polynomial systems, 
when the Cayley resultant method is equivalent to Cramer's rule, fall into this category. 
In the relative sense, it may be possible to show that the hidden variable resultant method 
based on Cayley or Sylvester is either numerically unstable during the construction of 
the resultant matrix or the resultant matrix has an eigenvalue that is more 
sensitive to small relative
perturbations than hoped. We do not know yet how to make such a statement precise. 

Instead, we provide an example that shows that 
the hidden variable resultant method remains numerically unstable 
in the relative sense.  Let $u$ be a sufficiently small real positive 
parameter and $d\geq 2$. 
Consider the following polynomial system:  
\[
\begin{aligned} 
 p_{2i-1}(\underline{x}) &= x_{2i-1}^2 + u\left(\tfrac{\sqrt{2}}{2}x_{2i-1} + \tfrac{\sqrt{2}}{2}x_{2i}\right),\\
 p_{2i}(\underline{x}) &= x_{2i}^2 + u\left(\tfrac{\sqrt{2}}{2}x_{2i} - \tfrac{\sqrt{2}}{2}x_{2i-1}\right), \qquad 1\leq i\leq \lfloor d/2\rfloor,
 \end{aligned} 
\]
where if $d$ is odd then take $p_d(\underline{x}) = x_d^2 + ux_d$. Selecting $\Omega=[-1,1]^d$, we have that $\| p_i \|_\infty = 1 + \sqrt{2} u$ for $1\leq i\leq d$, 
except possibly $\| p_d \|_\infty = 1+u$ if $d$ is odd. 
It can be shown that the origin\footnote{By a change of variables, there is an analogous example with a solution anywhere in $[-1,1]^d$. 
}, $\underline{x}^{\ast}$, is a simple root, 
$\det(J(\underline{x}^{\ast})) = u^d$, 
$\|J(\underline{x}^{\ast})^{-1}\|_2 = u^{-1}$, and 
that
\[
 f_{Cayley}(s_1,\ldots,s_{d-1},t_1,\ldots,t_{d-1}) = \prod_{k=1}^{d-1} (s_k + t_k) x_d^2 + \mathcal{O}(u). 
\]
Thus, neither the polynomials $p_i$ or the resultant matrix 
$R_{Cayley}(x_d)$ are small. In such an example, the relative condition number
will exhibit the same behavior as the absolute condition number. In particular, 
the relative condition number of an eigenvalue of $R_{Cayley}(x_d)$ may be larger 
than the relative condition number of the corresponding solution by a factor 
that grows exponentially with $d$. 

The same example (for $d=2$), and a similar argument, applies 
to the Sylvester matrix showing the conditioning can 
be squared in the relative sense too.

\section{Future outlook}\label{sec:futureWork}
In this paper we have shown that two popular hidden variable 
resultant methods based on the Sylvester and Cayley matrices 
are numerically unstable. Our analysis is for degree-graded polynomial bases and does not include the Lagrange basis or certain sparse bases. We believe that the analysis of the Cayley matrix in Section~\ref{sec:Cayley} could be extended to include general polynomial bases, though the analysis in  Section~\ref{sec:Sylvester} for the Sylvester matrix is more intimately connected to degree-graded bases.  We hesitantly suggest that hidden variable resultant methods are inherently plagued by numerial instabilities, and that neither other polynomial bases nor other resultants can avoid a worst-case scenario that we have identified in this paper.
We do not know
exactly how to formulate such a general statement, but we note that 
practitioners are widely experiencing problems with 
hidden variable resultant methods. In particular, we do not
know of a numerical multidimensional rootfinder based on resultants 
that is robust for polynomial systems of large degree $n$ and high $d$. 

However, at the moment the analysis that we offer here is limited to the Cayley and Sylvester matrices. Despite our doubts that it exists, we would celebrate the 
discovery of a resultant matrix that can be constructed numerically 
and that provably does not lead to a numerically unstable hidden variable 
resultant method.  
This would be a breakthrough in global rootfinding 
with significant practical applications as it might allow~\eqref{eq:polynomialSystem}
to be converted into a large eigenproblem without confronting conditioning issues. 
Solving high-dimensional and large degree polynomial systems would then be 
restricted by computational cost rather than numerical accuracy. 

Finally, we express again our hope that this paper, while 
appearing rather negative, will have a positive long-term impact on future research into
numerical rootfinders. 
%


\section*{Acknowledgments} 
We thank Yuji Nakatsukasa, one of our closest colleagues, for his 
insightful discussions during the writing of~\cite{Nakatsukasa_15_01} that 
ultimately lead us to consider conditioning issues more closely. We also 
thank Anthony Austin and Martin Lotz for carefully reading a draft and providing us with 
excellent comments. 
While this manuscript was in a much earlier 
form Martin Lotz personally sent it to Gregorio Malajovich for his comments. 
Gregorio's comprehensive and enthusiastic reply encouraged us to proceed with
renewed vigor.

\appendix  
\section{A generalization of Clenshaw's algorithm for degree-graded polynomial bases}\label{sec:appendix} 
This appendix contains the tedious, though necessary, proofs required in 
Section~\ref{subsec:GeneralizedClenshaw} for Clenshaw's algorithm for 
evaluating polynomials expressed in a degree-graded basis. 

\begin{proof}[Proof of Lemma~\ref{lem:ClenshawDegreeGraded}]
 By rearranging~\eqref{eq:ClenshawLikeAlg} we have $a_k = b_k[p](x) - (\alpha_kx + \beta_k)b_{k+1}[p](x) - \sum_{j=k+1}^{n-1} \gamma_{j,k+1}b_{j+1}[p](x)$. Thus,
 \[
    p(x) = a_0\phi_0(x) + \sum_{k=1}^n \!\left[b_k[p](x) - (\alpha_kx+\beta_k)b_{k+1}[p](x) -\!\! \sum_{j=k+1}^{n-1} \gamma_{j,k+1}b_{j+1}[p](x)\right]\!\phi_k(x).
 \]
Now, by interchanging the summations and collecting terms we have 
 \[
 \begin{aligned}
  p(x) & = a_0\phi_0(x) + \sum_{k=1}^n\phi_k(x)b_k[p](x) - \sum_{k=2}^n \left(\alpha_{k-1}x+\beta_{k-1}\right)\phi_{k-1}(x)b_k[p](x)\\
  & \qquad\quad\qquad\qquad\qquad\qquad\qquad\qquad\qquad\qquad -\sum_{j=2}^{n-1}\left[\sum_{k=1}^{j-1}\gamma_{j,k+1}\phi_{k}(x)\right]b_{j+1}[p](x)\\
     & = a_0\phi_0(x) + \phi_1(x)b_1[p](x)\\
     &\qquad\qquad\qquad+\sum_{j=1}^{n-1}\left[\phi_{j+1}(x) -(\alpha_{j}x+\beta_{j})\phi_{j}(x) - \sum_{k=1}^{j-1} \gamma_{j,k+1}\phi_{k}(x)\right]b_{j+1}[p](x)\\
 \end{aligned} 
 \]
 Finally, using~\eqref{eq:degreeGradedRecurrence} we obtain 
 \[
  p(x)= a_0\phi_0(x) + \phi_1(x)b_1[p](x) + \sum_{j=1}^{n-1} \gamma_{j,1}\phi_0(x)b_{j+1}[p](x),
 \]
as required.
\end{proof}

Section~\ref{subsec:GeneralizedClenshaw} also shows that Clenshaw's algorithm 
connects to the quotient $(p(x)-p(y))/(x-y)$. To achieve this we need an 
immediate result that proves a different recurrence relation on the Clenshaw shifts
to~\eqref{eq:ClenshawLikeAlg}. The proof involves tedious algebraic manipulations 
and mathematical strong induction.
\begin{lemma} 
Let $n$ be an integer, $\phi_0,\ldots,\phi_n$ a degree-graded basis 
satisfying~\eqref{eq:degreeGradedRecurrence}, and $b_{n+1}[p],\ldots,b_1[p]$ the Clenshaw 
shifts satisfying~\eqref{eq:ClenshawLikeAlg}. Then, for $1\leq j\leq n$, 
 \[
  b_{j}[\phi_{n+1}](x) = (\alpha_nx+\beta_n)b_{j}[\phi_{n}](x) + \sum_{s=j+1}^n\gamma_{n,s}b_{j}[\phi_{s-1}](x).
 \]
 \label{lem:NewRecurrence}
\end{lemma}
\begin{proof}
We proceed by induction on $j$. Let $j = n$. We have, by~\eqref{eq:ClenshawLikeAlg}, 
\[
 b_n[\phi_{n+1}](x) = (\alpha_nx + \beta_n)b_{n+1}[\phi_{n+1}](x) = (\alpha_nx + \beta_n)b_{n}[\phi_{n}](x),
\]
where the last equality follows because $b_{n+1}[\phi_{n+1}](x) = b_n[\phi_n](x) = 1$. 
Now, suppose the result holds for $j = n,n-1,\ldots,k+1$. We have, by~\eqref{eq:ClenshawLikeAlg} and 
the inductive hypothesis,
\[
\begin{aligned} 
   b_k[\phi_{n+1}](x) &= (\alpha_kx + \beta_k)b_{k+1}[\phi_{n+1}](x) + \sum_{j=k+1}^{n} \gamma_{j,k+1}b_{j+1}[\phi_{n+1}](x)\\
   & = (\alpha_kx + \beta_k)\left[(\alpha_nx+\beta_n)b_{k+1}[\phi_{n}](x) + \sum_{s=k+2}^n \gamma_{n,s}b_{k+1}[\phi_{s-1}](x)\right] \\
    &\hspace{1cm}+ \sum_{j=k+1}^{n-1} \gamma_{j,k+1}\left[(\alpha_nx+\beta_n)b_{j+1}[\phi_{n}](x) + \sum_{s=j+2}^n \gamma_{n,s}b_{j+1}[\phi_{s-1}](x)\right] \\
    &\hspace{2cm}+ \gamma_{n,k+1}b_{n+1}[\phi_{n+1}](x).\\
\end{aligned}
\]
By interchanging the summations and collecting terms we have 
\[
\begin{aligned} 
 b_k[\phi_{n+1}](x) & = (\alpha_nx+\beta_n)\left[(\alpha_kx + \beta_k)b_{k+1}[\phi_n](x) + \sum_{j=k+1}^{n-1}\gamma_{j,k+1}b_{j+1}[\phi_n](x)\right] \\
    & \hspace{.6cm} + \sum_{s=k+3}^n \gamma_{n,s}\left[(\alpha_kx + \beta_k)b_{k+1}[\phi_{s-1}](x) + \sum_{j=k+1}^{s-2}\gamma_{j,k+1}b_{j+1}[\phi_{s-1}](x)\right]\\
    & \hspace{2cm}+ \gamma_{n,k+2}(\alpha_kx+\beta_k)b_{k+1}[\phi_{k+1}](x) + \gamma_{n,k+1}b_{n+1}[\phi_{n+1}](x)\\
    & = (\alpha_nx+\beta_n)b_k[\phi_n] + \sum_{s=k+1}^n \gamma_{n,s} b_k[\phi_{s-1}],
\end{aligned}
\]
where in the last equality we used~\eqref{eq:ClenshawLikeAlg}, 
$(\alpha_kx + \beta_k)b_{k+1}[\phi_{k+1}](x) = b_k[\phi_{k+1}](x)$, and 
$b_{n+1}[\phi_{n+1}](x) = b_k[\phi_k](x) = 1$. 
\end{proof}

The recurrence from Lemma~\ref{lem:NewRecurrence} allows us to prove Theorem~\ref{thm:diffDegreeGraded}. 
\begin{proof}[Proof of Theorem~\ref{thm:diffDegreeGraded}] 

\textbf{Case 1: $\mathbf{x\neq y}$.}
Since for a fixed $y$ the Clenshaw shifts are linear, i.e., $b_j[c_1\phi_i+c_2\phi_k](y) = c_1b_j[\phi_i](y) +c_2 b_j[\phi_k](y)$ for constants $c_1$ and $c_2$, it 
is sufficient to prove the theorem for $p=\phi_n$ for $n\geq 1$.

We proceed by induction on $n$. For $n = 1$ we have 
\[
\sum_{j=0}^{n-1} \alpha_j b_{j+1}[\phi_{n+1}](y)\phi_j = \alpha_0 b_1[\phi_1](y) = \alpha_0 = \frac{\phi_1(x)-\phi_1(y)}{x-y}.
\]

Assume that the result holds for $n = 1,\ldots, k-1$. From 
the inductive hypothesis, we have
\[
\begin{aligned}
\!\frac{\phi_{k+1}(x) - \phi_{k+1}(y)}{x-y} &= \alpha_k\phi_k(x) + (\alpha_kx + \beta_k)\frac{\phi_{k}(x) - \phi_{k}(y)}{x-y}  \\
& \qquad\qquad\qquad\qquad\qquad\qquad+ \sum_{j=1}^k\!\gamma_{k,j}\frac{\phi_{j-1}(x) - \phi_{j-1}(y)}{x-y}\\
& = \alpha_k\phi_k(x) + (\alpha_kx+\beta_k)\sum_{j=0}^{k-1} \alpha_jb_{j+1}[\phi_k](y)\phi_j(x) \\
& \qquad\qquad\qquad\qquad\qquad\qquad+ \sum_{j=1}^k\gamma_{k,j}\sum_{s=0}^{j-2}\alpha_sb_{s+1}[\phi_{j-1}](y)\phi_s(x).
\end{aligned} 
\]
Moreover, by interchanging the summations and collecting terms we have 
\[
\begin{aligned} 
\frac{\phi_{k+1}(x) - \phi_{k+1}(y)}{x-y} & = \alpha_k\phi_k(x) + (\alpha_kx+\beta_k)\alpha_{k-1}b_k[\phi_k](y)\phi_{k-1}(x) \\
&\!\!\!+ \sum_{j=0}^{k-2}\alpha_j\left[(\alpha_kx+\beta_k)b_{j+1}[\phi_k](y) + \sum_{s=j+2}^k\gamma_{k,s}b_{j+1}[\phi_{s-1}](y)\right]\phi_j(x).
\end{aligned} 
\]
Finally, since $b_{k+1}[\phi_{k+1}](y) = 1$, $b_{k}[\phi_{k+1}](y) = (\alpha_kx+\beta_k)b_k[\phi_k](y)$, and by~\eqref{eq:ClenshawLikeAlg}, we have
\[
\begin{aligned}
 \frac{\phi_{k+1}(x) - \phi_{k+1}(y)}{x-y} &= \alpha_kb_{k+1}[\phi_{k+1}](y)\phi_k(x) + \alpha_{k-1}b_{k}[\phi_{k+1}](y)\phi_{k-1}(x) \\ 
 & \qquad\qquad\qquad + \sum_{j=0}^{k-2} \alpha_jb_{j+1}[\phi_{k+1}](y)\phi_j(x)
\end{aligned}
\]
and the result follows by induction. 

\textbf{Case 2: $\mathbf{x = y}$.}
 Immediately follows from $x\neq y$ by using L'Hospital's rule 
 on~\eqref{eq:diffDegreeGraded}. 
\end{proof}

\end{document}